\documentclass[11pt,psamsfonts]{amsart}
\usepackage[margin=1.5in]{geometry}
\usepackage{wrapfig}
\usepackage{amssymb}
\usepackage{geometry}

\usepackage{float}

\usepackage{amsmath,amsfonts,amsthm,amstext}
\usepackage[latin1]{inputenc}
\usepackage{fancybox}
\usepackage{niceframe}
\usepackage{array}
\usepackage{newlfont}
\usepackage{verbatim}
\usepackage[dvips]{psfrag}
\usepackage{color}
\usepackage{float}
\usepackage{indentfirst}
\usepackage[normalem]{ulem}
\usepackage{pst-text}
\usepackage{graphicx}
\usepackage{graphics}
\usepackage{overpic}
\usepackage{hyperref}
\usepackage{pst-node}
\usepackage{tikz-cd} 
\usepackage{enumerate}

\graphicspath{{./Figures/}}
\usepackage{wrapfig}
\newcommand{\R}{\ensuremath{\mathbb{R}}}

\newcommand{\N}{\ensuremath{\mathbb{N}}}

\newcommand{\s}{\Sigma}

\newcommand{\la}{\lambda}

\newcommand{\e}{\varepsilon}

\newcommand{\V}{\mathcal{V}}

\newcommand{\Cr}{\mathcal{C}^{r}}
\newcommand{\Xr}{\chi^{r}}
\newcommand{\Or}{\Omega^{r}}
\newcommand{\rn}[1]{\mathbb{R}^{#1}}
\newcommand{\U}{\mathcal{U}}
\newcommand{\er}{\mathcal{O}}
\newcommand{\p}{\varphi}
\newcommand{\ag}{\alpha}
\newcommand{\bg}{\beta}
\newcommand{\cg}{\gamma}
\newcommand{\dg}{\delta}
\newcommand{\sgn}{\textrm{sgn}}

\newtheorem {theorem} {Theorem} 
\newtheorem {prop}  {Proposition}
\newtheorem {corollary}  {Corollary}
\newtheorem {lemma}  {Lemma}
\newtheorem {definition}  {Definition}
\newtheorem {remark}  {Remark}
\newtheorem {example} {Example}

\makeatletter
\DeclareFontFamily{U}{tipa}{}
\DeclareFontShape{U}{tipa}{m}{n}{<->tipa10}{}
\newcommand{\arc@char}{{\usefont{U}{tipa}{m}{n}\symbol{62}}}%

\newcommand{\arc}[1]{\mathpalette\arc@arc{#1}}

\newcommand{\arc@arc}[2]{%
	\sbox0{$\m@th#1#2$}%
	\vbox{
		\hbox{\resizebox{\wd0}{\height}{\arc@char}}
		\nointerlineskip
		\box0
	}%
}
\makeatother

\newcommand{\arcf}[3]{\arc{#1#2}|_{#3}}

\definecolor{verde}{rgb}{0.0,0.5,0.0}
\definecolor{azul}{rgb}{0,0,128}
\definecolor{roxo}{rgb}{0.44,0.16,0.39}
\definecolor{vinho}{rgb}{0.5,0.0,0.13}
\definecolor{lilas1}{rgb}{0.6,0.33,0.73}
\definecolor{rosa}{rgb}{0.84,0.04,0.33}
\definecolor{mostarda}{rgb}{0.91,0.41,0.17}
\definecolor{mostarda2}{rgb}{1.0,0.66,0.07}

\newtheorem {mtheorem} {Theorem} 

\usepackage{mathpazo}
\begin{document}

\title[Bifurcation Diagrams of Global Connections in Filippov Systems]{Bifurcation Diagrams of Global Connections\\ in Filippov Systems}

\author[K. S. Andrade]{Kamila S. Andrade}
\address[KSA]{Departamento de Matem\'{a}tica, Instituto de Matem\'{a}tica e Estat\'{i}stica  (IME),  Universidade Federal de Goi\'{a}s (UFG), 74690-900, Goi\^ania, GO, Brazil.}
\email{kamila.andrade@ufg.br}

\author[O. M. L. Gomide]{Ot\'avio M. L. Gomide}
\address[OMLG]{Departamento de Matem\'{a}tica, Instituto de Matem\'{a}tica e Estat\'{i}stica  (IME),  Universidade Federal de Goi\'{a}s (UFG), 74690-900, Goi\^ania, GO, Brazil.}
\email{otaviomarcal@ufg.br}

\author[D. D. Novaes]{Douglas D. Novaes}
\address[DDN]{Departamento de Matem\'{a}tica, Instituto de Matem\'{a}tica, Estat\'{i}stica e Computa\c{c}\~{a}o Cient\'{i}fica (IMECC), Universidade
Estadual de Campinas (UNICAMP), Rua S\'{e}rgio Buarque de Holanda, 651, Cidade Universit\'{a}ria Zeferino Vaz, 13083--859, Campinas, SP,
Brazil.}
\email{ddnovaes@unicamp.br}



\maketitle


\begin{abstract}
In this paper, we are concerned about the qualitative behavior of planar Filippov systems around some typical invariant sets, namely, polycycles. In the smooth context, a polycycle is a simple closed curve composed by a collection of singularities and regular orbits, inducing a first return map.  Here, this concept is extended to Filippov systems by allowing typical Filippov singularities lying on the switching manifold. Our main goal consists in developing a method to investigate the unfolding of polycycles in Filippov systems. In addition, we apply this method to describe bifurcation diagrams of Filippov systems around certain polycycles.  
\end{abstract}

\section{Introduction and statement of the main results}

In 1882, the concept of limit cycle was introduced by Henri Poincar\'{e} and, since then, the detection of such an object has become one of the most interesting and complicated problems in the Qualitative Theory of Dynamical System. Over the years, other global structures were investigated, and the concept of polycycle have been established. Roughly speaking, a polycycle is a simple closed curve composed by a collection of singularities and regular orbits, inducing a first return map. This class of invariant sets has been extensively studied in the literature as in the so-called \textit{Dulac's Problem}.

In the present paper, we aim to study such objects in the context of planar Filippov systems, that is, planar piecewise smooth systems having their trajectories ruled by the Filippov's convention \cite{F} (its formal definition will be introduced right below). We are mainly motivated by classical and recent studies that addressed a special kind of polycycles, namely the homoclinic ones, that is, a  regular trajectory connecting  a singular  point  to itself. 
 
In the last years, homoclinic-like polycycles of planar Filippov systems have received attention of the mathematical community. In \cite{KRG}, Kuznetsov et al. provided a catalog of bifurcations occurring in one-parameter families of Filippov systems. Among them, they presented the \textit{critical crossing cycle bifurcation} ($CC$-bifurcation), which consists in a one-parameter family $Z_{\ag}$ of Filippov systems, for which $Z_0$ has a homoclinic-like polycycle at a fold-regular singularity.  In \cite{GTS},  by means of Bifurcation Theory, Guardia et al. approached the $CC$-bifurcation phenomenon in a more general setting than the one presented in \cite{KRG}. Finally, in \cite{FPT15}, Freire et al. showed that the unfolding of a $CC$-bifurcation provided by \cite{KRG} holds in a generic scenario.  It is worth mentioning that such a global phenomenon has already appeared in the local unfolding of $\s$-singularities with higher degeneracies \cite{BCT12}. 

Recently, more degenerated homoclinic-like polycycles through $\s$-singularities were considered. In \cite{NTZ18}, Novaes et al. studied a codimension-two homoclinic-like polycycle at a visible--visible fold-fold singularity and provided its complete bifurcation diagram. In \cite{A16,AJMT17}, Andrade et al.  studied   a class of systems presenting a homoclinic-like polycycle at a saddle-regular singularity (also know as \textit{boundary-saddle singularity}), they also described some bifurcations and a physical model realizing such a connection. Homoclinic-like polycycles of Filippov systems has also been considered in the context of regularization process (see, for instance, \cite{BS16,NR21,NR22}).

Polycycles through more than one $\s$-singularity have also appeared in the literature. For instance, in \cite{BPET18}, Benadero et al. studied a nonsmooth model of electronic circuits with power inverters admitting a polycycle passing through two fold-regular singularities. Other examples of polycycles through  $\s$-singularities appeared in \cite{BPET18, Fei1,Fei2,LH13, LW17, LR14}.  

Here, we shall develop a rather general method  to deal with polycycles going through tangential singularities. More specifically,  a mechanism will be developed for detecting crossing bifurcation phenomena. We then apply it to obtain the complete bifurcation diagram of Filippov systems around certain polycycles.

Before presenting our main results in the end of this section, we introduce the formal definition of  Filippov  systems and some basic concepts needed  for defining the class of  polycycles we shall consider.

\subsection{Filippov systems}\label{prem}

In the theory of nonsmooth dynamical systems, the notion of solutions of a piecewise smooth differential system 
\begin{equation}\label{eq:filippov}
Z(p)=\left\{\begin{array}{ll}
X(p),&h(p)>0,\\
Y(p),&h(p)<0,
\end{array}\right.\quad p\in M,
\end{equation}
is stated by the Filippov's convention (see \cite{F}). In this way, \eqref{eq:filippov} is called by Filippov system. Here, $M$ is an open bounded  connected set of $\rn{2}$ and  $h:M\rightarrow \rn{}$ is a smooth function having $0$ as a regular value. Therefore, $\Sigma=h^{-1}(0)$ is an embedded codimension one submanifold of $M$ which splits it in the sets $M^{\pm}=\{p\in M; \pm h(p)>0\}$.
The Filippov system \eqref{eq:filippov} is concisely denoted by $Z=(X,Y).$  The space of Filippov vector fields  $Z=(X,Y),$ where $X$ and $Y$  are $\Cr$ vector fields, is denoted by $\Omega^r.$

In order to illustrate the Filippov's convention, we distinguish the following open regions of the switching manifold $\Sigma$: (Crossing Region) $\Sigma^{c}=\{p\in \Sigma;\ Xh(p)Yh(p)>0\}$, (Stable Sliding Region) $\Sigma^{ss}=\{p\in \Sigma;\ Xh(p)<0,\ Yh(p)>0\},$ and (Unstable Sliding Region) $\Sigma^{us}=\{p\in \Sigma;\ Xh(p)>0,\ Yh(p)<0\},$ where $Xh(p)$ denotes the \textbf{Lie derivative} of $h$ in the direction of the vector field $X\in \Xr$ at $p\in\Sigma$ and is defined as $Xh(p)=\langle X(p), \nabla h(p)\rangle$. 
	
The local solution of $Z=(X,Y)\in \Or$ at $p\in\Sigma^c$ is given by the concatenation of the local solutions of $X|_{M^+}$ and $Y|_{M^+}$ at $p$. The local solution of $Z=(X,Y)\in \Or$ at $p\in\Sigma^c$ is given by the \textbf{sliding vector field}
\begin{equation*}
F_{Z}(p)=\frac{1}{Yh(p)-Xh(p)}\left(Yh(p)X(p)-Xh(p)Y(p)\right).
\end{equation*}
Notice that $F_{Z}$ is a vector field tangent to $\s^{s}$. The singularities of $F_{Z}$ in $\s^{s}$ are called \textbf{pseudo-equilibria} of $Z$.

The \textbf{tangency set} between $X$ and $\s$ is given by $S_{X}=\{p\in\Sigma;\ Xh(p)=0\}$. Accordingly, the tangency set of $Z$ will be referred as $S_{Z}=S_{X}\cup S_{Y}$. Notice that $\s$ is the disjoint union $\s^{c}\cup \s^{ss}\cup \s^{us}\cup S_{Z}$. Herein, $\s^{s}=\s^{ss}\cup \s^{us}$ is called \textbf{sliding region} of $Z$.  A point $p\in S_Z$ is called \textbf{tangential singularity} of $Z$ provided that $X(p), Y(p)\neq 0$;

We say that $p\in \Sigma$  is a \textbf{$\s$-singularity} of $Z$ provided that $p$ is either a tangential singularity, an equilibrium of $X$ or $Y$, or a pseudo-equilibrium of $Z$.

\begin{definition}
	 $X\in\Xr$ has an \textbf{$n$-multiplicity contact} (or $(n-1)$-order contact) with $\s$ at $p$ if $X^{i}h(p)=0$, for $i=1,\cdots, n-1$, and $X^{n}h(p)\neq 0$. In particular, for $n=2,3$, $p$ is said to be a \textbf{fold point} and  \textbf{cusp point} of $X,$ respectively.
\end{definition}

\begin{definition}\label{tangs}
Let $p\in\s$ be a tangential singularity of $Z=(X,Y)$, we say that $p$ is:
\begin{enumerate}[i)]
\item a {\bf regular-tangential singularity of multiplicity $n$} of $Z$ provided that $X$ (resp. $Y$) has a $n$-multiplicity contact with $\Sigma$ at $p$ and $Yh(p)\neq0$ (resp. $Xh(p)\neq0$);

\item a \textbf{tangential--tangential singularity} of $Z$ provided that $Xh(p)=Yh(p)=0$.
\end{enumerate}
\end{definition}

\begin{remark}
In  Definition \ref{tangs} i), for $n=2$ and $n=3,$ $p$ is said to be a \textbf{regular-fold singularity} and  \textbf{regular-cusp singularity} of $Z$, respectively.
\end{remark}

In the Filippov context, special attention must be paid to some singularities lying on $\Sigma,$ known as $\Sigma$-singularities, which also present local invariant manifolds.

Now, motivated by \cite{GTS}, we define the concept of local separatrix at a point $p\in\s$, which will play an important role in this paper. 
\begin{definition}
If $p\in\s$, the \textbf{stable (unstable) separatrix} $W^{s}_{\pm}(p)$ ($W^{u}_{\pm}(p)$) of $Z=(X_{+},X_{-})$ at a tangential singularity $p$ in $\s^{\pm}$ is defined as
	$$W^{s,u}_{\pm}(p)=\{q=\p_{X_{\pm}}(t(q),p);\ \p_{X_{\pm}}(I(q),p)\subset \s^{\pm}\ \textrm{and } \dg_{s,u}t(q)>0\},$$
	where, $\p_{X_{\pm}}$ is the flow of $X_{\pm}$, $\dg_{u}=1$, $\dg_{s}=-1$, and $I(q)$ is the open interval with extrema $0$ and $t(q)$.
\end{definition}

\subsection{Polycycles of Filippov systems}
The $\s$-singularities above admit global connections, which have no counterpart in the smooth context. In this way, the concept of polycycle can carried to Filippov systems. 
 In the next definition, we establish a variation of the classical concept of polycycle for Filippov systems.

\begin{definition}\label{def_scyclegeneralized}
	A closed curve $\Gamma$ is said to be a \textbf{polycycle} of $Z=(X,Y)$ if it is composed by a finite number of points, $p_1,p_2,\ldots,p_n$ and a finite number of regular orbits  of $Z$, $\gamma_1,\gamma_2,\ldots,\gamma_n$, such that for each $1\leq i\leq n$, $\gamma_{i}$ has ending points $p_i$ and $p_{i+1}$, where $p_{n+1}=p_1$. Moreover:
	\begin{enumerate}[i)]
		\item $\Gamma$ is homeomorphic to  $S^{1}$
		and it is oriented by increasing time along the regular orbits;
		\item if $p_{i}\in\s$ then it is a $\s$-singularity;
		\item if $p_{i}\in M^{\pm}$ then it is an equilibrium of either $X\big|_{M^+}$ or $Y\big |_{M^{-}}$;
		\item there exists a non-constant first return map defined, at least, in one side of $\Gamma$.
	\end{enumerate}
	In particular, if $p_{i}\in\s$, for all $1\leq i\leq n$, then $\Gamma$ is said to be a \textbf{$\s$-polycycle} which is referred as a \textbf{regular-tangential $\s$-polycycle} or a \textbf{tangential--tangential $\s$-polycycle} when all the $\s$-singularities of $Z$ contained in $\Gamma$ are regular-tangential singularities or tangential--tangential singularities, respectively.
\end{definition}

In the definition above, the {\bf homoclinic-like} polycycles mentioned above corresponds to the case $n=1$.

\subsection{Main results}

In what follows, we provide a briefly description of the results contained in this paper.

As usual, we say that $\gamma$ is a regular orbit of $Z=(X,Y)$ if it is a piecewise smooth curve such that $\gamma\cap M^{+}$ and $\gamma\cap M^{-}$ are unions of regular orbits of $X$ and $Y$, respectively, and  $\gamma\cap\s\subset\s^{c}$. In this case, $\partial \gamma$ is referred as the \textbf{ending points} of $\gamma$. Accordingly, a \textbf{cycle} is a closed regular orbit $\Gamma$ of $Z$. If $\Gamma\cap \s\neq \emptyset$, then $\Gamma$ is called a \textbf{crossing cycle} of $Z$.

One of our main goals in this paper is to characterize qualitatively the systems in a neighborhood of a closed curve. To do this we introduce the following notion on equivalence at a compact set. 

\begin{definition}\label{equiv}
	Let $\mathcal{K}$ be a compact set of $M$. We say that $Z$ and $Z_0$ are (topologically) \textbf{equivalent} at $\mathcal{K}$ if there exist neighborhoods $U$ and $V$ of $\mathcal{K}$ and an orientation preserving homeomorphism $h:U\rightarrow V$ which carries orbits of $Z$ onto orbits of $Z_0$. 
\end{definition}

 Following the techniques used in \cite{AJMT17,NTZ18}, we develop a mechanism, named \textit{Method of Displacement Functions} (see Section \ref{approach_sec}), to study the unfolding of $\s$-polycycles in a typical scenario.

Generally speaking, given a Filippov system $Z_0$ having a $\s$-polycycle $\Gamma_0$, the proposed method associates each $Z$ near $Z_0$ to a system of nonlinear equations, called {\bf crossing system}, which provides information on the crossing orbits of $Z$ in a neighborhood of $\Gamma_0$.  This system depends smoothly on $Z$ and is called {crossing system}.

Next result concerns about the crossing dynamics that bifurcates from a \(\s\)-polycycle. More specifically, it establishes that if $\gamma_0$ is a $\s$-polycycle of  $Z_0\in\Or$ and \(\mathcal{A}\) a neighborhood of \(\Gamma_0\) then the crossing dynamics, in \(\mathcal{A}\), of small perturbations of \(Z_0\) are totally characterized by the \(\s\)-singularities contained in \(\Gamma_0\).  
\begin{mtheorem}\label{crossingthm}
Let $Z_0\in\Or$ having a $\s$-polycycle $\Gamma_0$. There exist neighborhoods $\V$ and $\mathcal{A}$ of $Z_0$ in $\Or$ and $\Gamma_0$ in $\rn{2}$, respectively, in such a way that the crossing system associated to $Z$ is defined in $\mathcal{A}$ and it is completely characterized by the types of $\s$-singularities appearing in $\Gamma_0$.
\end{mtheorem}

Theorem \ref{crossingthm} is a consequence of Theorems \ref{charac1} and \ref{charac2} stated and proved in Section \ref{charac_sec}. Furthermore, Theorems \ref{charac1} and \ref{charac2} give more details on the characterization of the crossing system, nevertheless their statements require some technical constructions which are made in Section \ref{approach_sec}, and for this reason, we omit the details of the characterization in Theorem \ref{crossingthm} and invite the reader to visit these propositions.

Now we use Theorem \ref{crossingthm} to obtain a complete description of the bifurcation diagrams of certain $\s$-polycycles. First, we study the unfolding of $\s$-polycycles admitting a unique $\s$-singularity of regular-cusp type.  Recall that $Z_0=(X_0,Y_0)$ has a regular-cusp singularity at $p_0\in\s$ if $X_0$ has a contact of multiplicity $3$ with $\s$ at $p_0$ and $Y_0$ is transverse to $\s$ at $p_0$, or vice-versa. Denote by $\Omega_{RC}\subset \Or$ the class of Filippov systems admitting a $\s$-polycycle having a unique $\s$-singularity of regular-cusp type. 

Next result provides the bifurcation diagram of $Z_0\in\Omega_{RC},$ around the \(\s\)-polycycle \(\Gamma_0\). It is better detailed in Theorem \ref{regcuspcycle} which will be established in Section \ref{cusp_sec}. Roughly speaking, this theorem guarantees that crossing limit cycles generically bifurcate  from \(\Gamma_0\) as well as other types of \(\s\)-polycycles through fold-regular singularities having or not sliding segments.  Moreover, it shows that \(\Gamma_0\) is a global connection of codimension two and describes all codimension zero and one phenomena that occurs, near \(\Gamma_0\), for vector fields near \(Z_0\) in \(\Or\).

\begin{mtheorem} \label{formalregcuspcycle}
Given $Z_0\in\Omega_{RC},$  there exist of neighborhoods $\V\subset \Or$ of $Z_0$ and $V\subset\R^2$ of the origin and a surjective function $(\beta,\lambda_1):\V\rightarrow V,$ with $(\beta,\lambda_1)(Z_0)=(0,0)$, such that the parameters $(\beta,\lambda_1)$ completely describe the bifurcation diagram of $Z_0$ around its $\s$-polycycle given by Figure \ref{bifcr1}.
\end{mtheorem}

Theorem \ref{formalregcuspcycle} is equivalent to Theorem \ref{regcuspcycle} of Section \ref{cusp_sec} which contains the complete description of the bifurcation diagram presented in Figure \ref{bifcr1}.

\begin{figure}[H]
	\centering
	\bigskip
	\begin{overpic}[width=13cm]{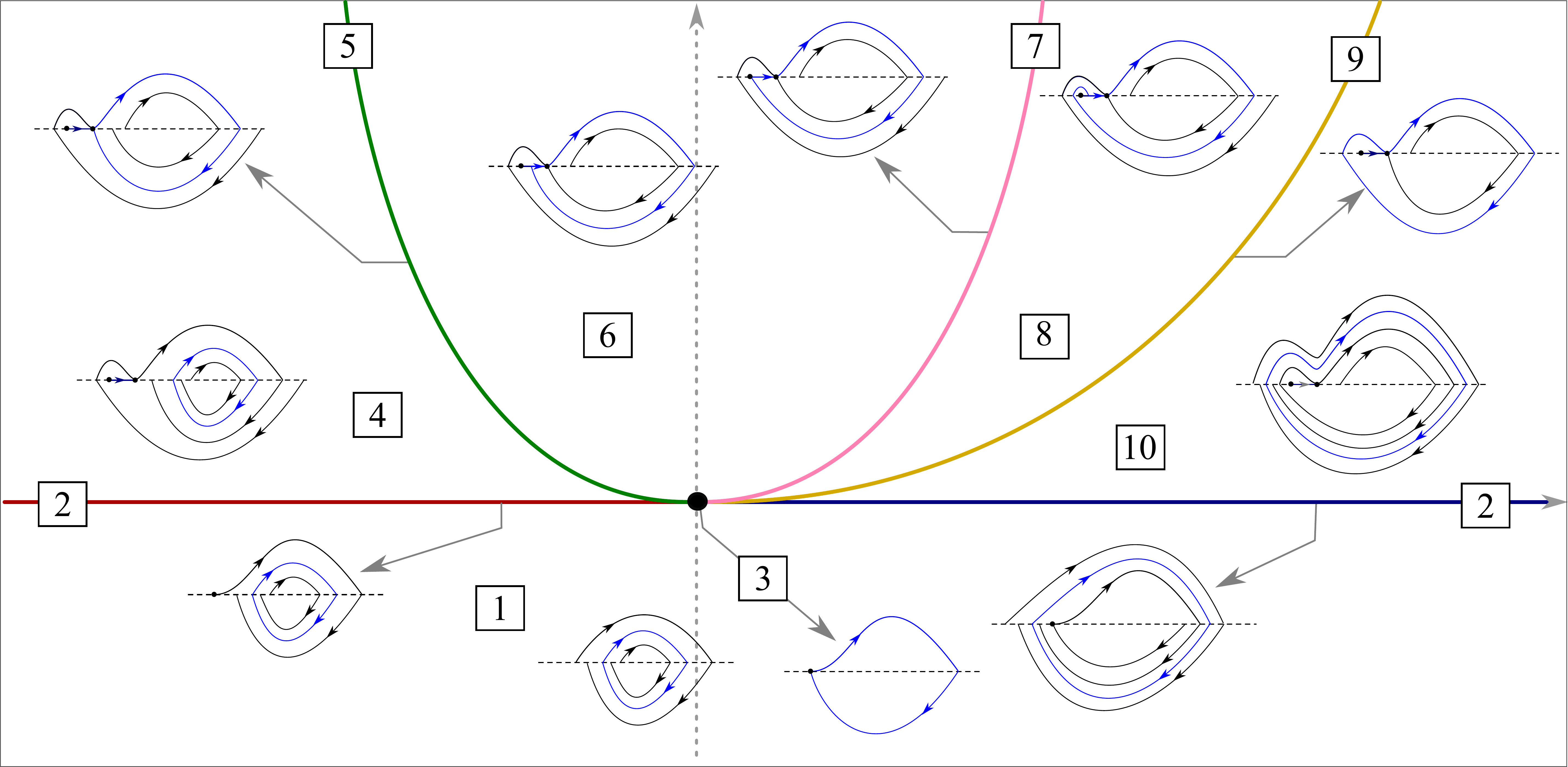}
		\put(97.5,14.5){{\footnotesize $\bg$}}
		\put(46,47){{\footnotesize $\lambda_1$}}
		\put(24,42){{\footnotesize $\overline{V}$}}
		\put(63,40){{\footnotesize $\overline{I}$}}
		\put(77,35){{\footnotesize $\overline{A}$}}
	\end{overpic}
	\bigskip
	\caption{Bifurcation diagram of $Z_0\in\Omega_{RC}$ around $\Gamma_0$. $\overline{V}$, $\overline{I}$, $\overline{A}$ and the $\beta$-axis are codimension-one bifurcation curves. All the crossing limit cycles appearing in the diagram are hyperbolic.} \label{bifcr1}
\end{figure} 

For a practical model realizing such bifurcation diagram see Example \ref{example1}.

In light of the extensively studied critical crossing cycle bifurcation, we consider a generalization of such a $\s$-polycycle. More specifically, we allow the $\s$-polycycle to have two $\s$-singularities of fold-regular type, instead of only one. Denote by $\Omega_{DRF}\subset \Or$ the class of Filippov systems admitting a $\s$-polycycle $\Gamma$ having exactly two $\s$-singularities, $p_1$ and $p_2,$ satisfying
\begin{enumerate}[i)]
	\item $p_1$ and $p_2$ are regular-fold singularities of $Z_0$;
	\item there exist two curves $\gamma_1$ and $\gamma_2$ connecting $p_1$ and $p_2$, oriented from $p_1$ to $p_{2}$ and from $p_2$ to $p_1$, respectively, 
	such that $\Gamma=\gamma_1\cup\gamma_2$, $\gamma_1$ is tangent to $\s$ at $p_1$ and transverse to $\s$ at $p_2$, and $\gamma_2$ is tangent to $\s$ at $p_2$ and transverse to $\s$ at $p_1$.
\end{enumerate}

Similarly to the previous theorem, next result is better detailed in Theorem \ref{bifdig_foldregular} which will be established in Section \ref{2fold-fold}. Roughly speaking, if  $Z_0\in\Omega_{DRF}$ with the \(\s\)-polycycle \(\Gamma_0\) as stated above, next theorem guarantees that crossing limit cycles generically bifurcate  from \(\Gamma_0\) as well as other types of \(\s\)-polycycles through one or two fold-regular singularities having or not sliding segments. Moreover, it shows that \(\Gamma_0\) is a global connection of codimension two and describes all codimension one and zero phenomena that occurs, near \(\Gamma_0\), for vector fields near \(Z_0\) in \(\Or\).

\begin{mtheorem}\label{formalbifdig_foldregular}
	Given $Z_0\in\Omega_{DRF},$ there exist neighborhoods $\V\subset \Or$ of $Z_0$ and $V\subset \R^2$ of the origin and a surjective function $(\beta_1,\beta_2):\V\rightarrow V$ with $(\beta_1,\beta_2)(Z_0)=(0,0)$, such that the parameters $(\beta_1,\beta_2)$ completely describe the bifurcation diagram of $Z_0$ around its $\s$-polycycle $\Gamma$ given by Figure \ref{bif2fr1}.
\end{mtheorem}

Theorem \ref{formalbifdig_foldregular}  is equivalent to Theorem \ref{bifdig_foldregular} of Section \ref{2fold-fold} which contains the complete description of the bifurcation diagram presented in Figure \ref{bif2fr1}.

\begin{figure}[H]
	\centering
	\bigskip
	\begin{overpic}[width=13cm]{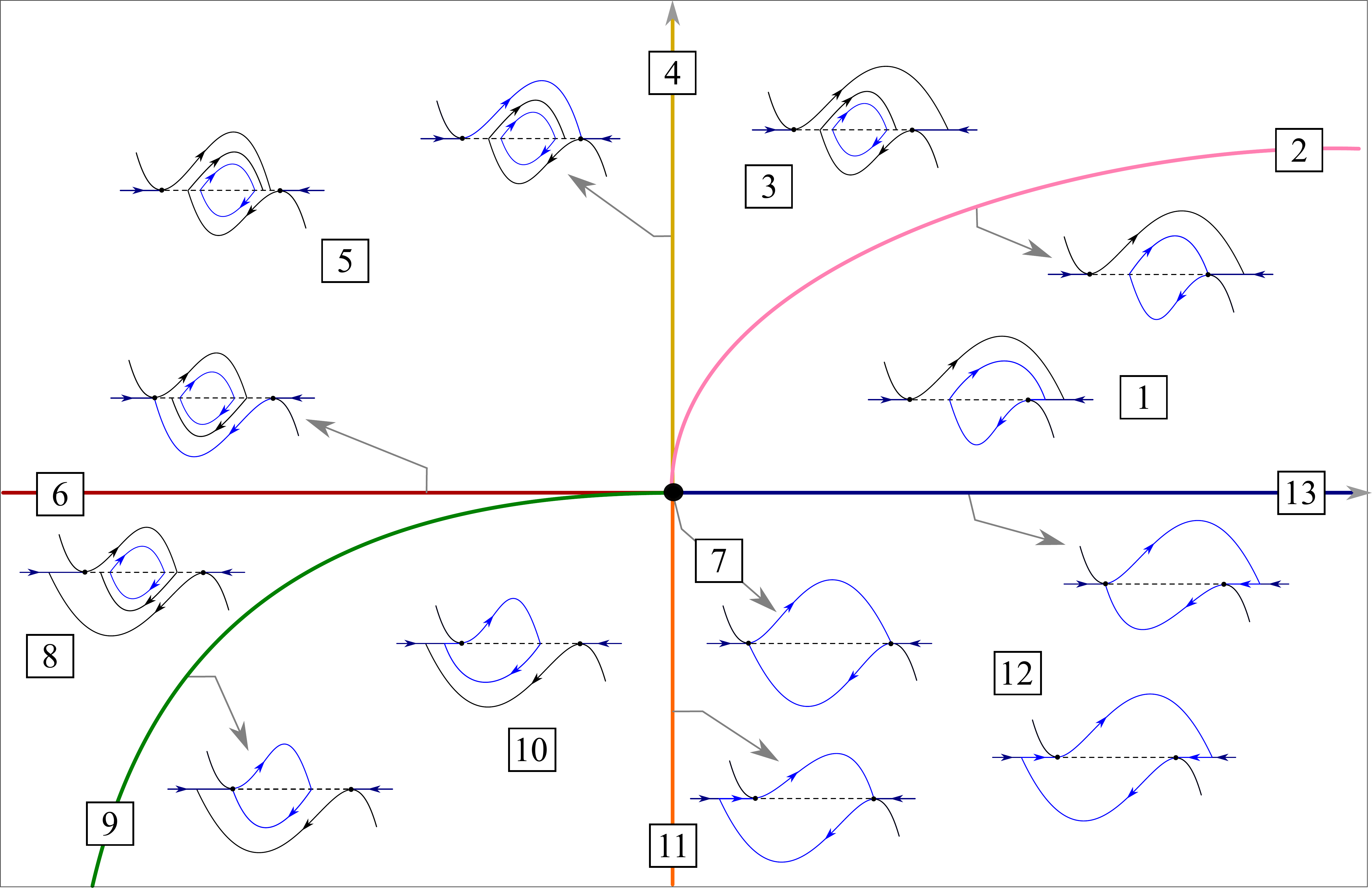}
		\put(97,26){{\footnotesize $\bg_1$}}
		\put(50,62){{\footnotesize $\bg_2$}}
		\put(85,54){{\footnotesize $\cg_1$}}
		\put(6,8){{\footnotesize $\cg_2$}}
	\end{overpic}
	\bigskip
	\caption{ Bifurcation diagram of $Z_0\in\Omega_{DRF}$ around $\Gamma_0$. In this case, $\cg_1$, $\cg_2$, the $\beta_1$-axis and the $\beta_2$-axis are codimension-one bifurcation curves. All the crossing limit cycles appearing in the diagram are hyperbolic.}	\label{bif2fr1}
\end{figure} 	

A practical model realizing such bifurcation diagram can be found in \cite[Example 2 ]{Wu22}.

\subsection{Structure of the paper}The paper is organized as follows. In Section \ref{approach_sec}, we develop the method of displacement functions which makes use of transition maps, mirror maps and displacement functions introduced in Sections \ref{trans-sec}, \ref{involution_sec} and \ref{disp_sec}, respectively. Section \ref{charac_sec} is devoted to state and prove Theorems \ref{charac1} and \ref{charac2}, which characterizes the transition maps and proves Theorem \ref{crossingthm}. The $\s$-polycycles containing only $\s$-singularities of regular-tangential type are analyzed in Section \ref{reg-tan_sec}. More specifically, in Section \ref{descripregtan} we characterize the crossing system for such a class of $\s$-polycycles. In Section \ref{polyn_sec}, we prove general properties of $\s$-polycycles containing a unique $\s$-singularity of regular-tangential type.  In Section \ref{cusp_sec}, we state and prove Theorem \ref{regcuspcycle}. Section \ref{polyn+_sec} is devoted to extend the properties described in Section \ref{polyn_sec} to a wider class of systems.

\section{Method of Displacement Functions}\label{approach_sec}

The aim of this section is to provide a systematic methodology for studying aspects of structural stability of $\s$-polycycles in $2D$ nonsmooth vector fields via displacement functions as well as to describe the bifurcations of these objects.

In what follows, given a $\s$-polycycle $\Gamma_0$ of $Z_0\in\Or$, we outline the method developed in this work for detecting all the crossing limit cycles with the same topological type of $\Gamma_0$ bifurcating from $\Gamma_0$. By ``the same topological type'' we understand the cycles which can be continuously deformed into $\Gamma_0$ inside a small annulus $\mathcal{A}$ around $\Gamma_0$. In general, our method regards in reducing the problem of finding crossing limit cycles to the study a system of nonlinear equations. 

Assuming that the $\s$-polycycle $\Gamma_0$ contains $k$ $\s$-singularities $p_{i}$, $i=1,\cdots, k+1$ ($p_1=p_{k+1}$), the totality of this section is devoted to construct, for each nonsmooth vector field $Z\in\Or$ near $Z_0$, a displacement function (see Definition \ref{disp_def}) $\Delta_{i}(Z):\sigma_{i}(Z)\times\sigma_{i+1}(Z)\to \R$ which measures the splitting of the connection between $p_i$ and $p_{i+1}$ through $\Gamma_0$, for $i=1,\cdots,k$ (see Figure \ref{dispfig}). This allow us to introduce, for each nonsmooth vector field $Z\in\Or$ near $Z_0$, the \textbf{crossing system}:
\begin{equation}\label{cross}
\left\{
\begin{array}{l}
\vspace{0.2cm}\Delta_{1}(Z)(x_{1},x_{2})=0;\\
\vspace{0.2cm}\Delta_{2}(Z)(x_{2},x_{3})=0;\\
\ \ \ \ \ \ \ \ \ \ \vdots\\
\vspace{0.2cm}\Delta_{k-1}(Z)(x_{k-1},x_{k})=0;\\
\vspace{0.2cm}\Delta_{k}(Z)(x_{k},x_{1})=0;\\
x_{i}\in\sigma_{i}(Z),\ i=1,\cdots,k,
\end{array}
\right.
\end{equation}
We anticipate that the displacement functions $\Delta_i$ in \eqref{cross} will be given via transition maps and mirror maps while the domains $\sigma_{i}(Z)$ will be a finite union of real intervals such that $\mathcal{A}\cap\s^c\subset \cup_{i=1}^{k} \sigma_{i}(Z_0)$.  We shall see that each solution $x(Z)=(x_{1}(Z),\cdots,x_{k}(Z))$ of \eqref{cross} will correspond to a closed orbit $\Gamma(Z)$ of $Z$ contained in $\mathcal{A}$ satisfying $x_i(Z)=\Gamma(Z)\cap\sigma_i(Z)$, $i=1,\cdots,k$. In addition, if $x(Z)$ is an isolated solution of  \eqref{cross} such that $x_{i}(Z)\in\textrm{int}(\sigma_{i}(Z))$ for each $i=1,\cdots,k$, then it corresponds to a crossing limit cycle of $Z$. On the other hand, if there exists $i\in\{1,\cdots,k\}$ such that $x_{i}(Z)\in\partial \sigma_{i}(Z)$ then this solution corresponds to a $\s$-polycycle. Reciprocally, if $\Gamma$ is a closed orbit of $Z$ in $\mathcal{A}$ and $x_i=\Gamma\cap\sigma_i(Z)$  for $i=1,\cdots,k,$ then $(x_1,\ldots,x_k)$ is a solution of \eqref{cross}. Therefore, system \eqref{cross} describes the whole crossing dynamics of $Z$ in $\mathcal{A}$.

\subsection{Transition Maps} \label{trans-sec}

In order to understand the behavior of the  nonsmooth vector fields near $Z_0$ in $\mathcal{A}$ we shall study how the crossing trajectories of $Z_{0}$ behave near the $\s$-singularities in $\Gamma_0$. With this purpose, we establish a precise definition for transition maps at points $p\in\s$.

We shall see that a transition map is defined for each component, $X$ and $Y$, of a nonsmooth vector field $Z=(X,Y)$. In light of this, we consider a smooth vector field $X_0\in\Xr$ on $M$ and we study the behavior of its trajectories passing through the codimension one manifold $\Sigma\subset M$ given in Subsection \ref{prem}.

Assume that $X_0$ satisfies the following set of hypotheses \textbf{(T)} at a point $p_0\in \s$:
\begin{itemize}
	\item [($T_1$)] $X_0(p_0)\neq(0,0)$;
	\item [($T_2$)] there exists $t_0\in\R$ such that $q_{0}=\p_{X_0}(t_0;p_{0})\in M^{\pm}$,
\end{itemize}
where $\p_{X_0}$ denotes the flow of $X_0$. 

Let $\tau\subset M^{\pm}$ be a local transversal section of $X_0$ at $q_{0}$. From the Implicit Function Theorem for Banach Spaces there exist neighborhoods $\U_0\subset\chi^r$ of $X_0$ and $V_0\subset M$ of $p_0$, $\e>0,$ and a unique smooth function $s: \U_0\times V_0\rightarrow (t_0-\e,t_0+\e)$ such that $s(X_0,p_{0})=t_0$ and $\p_{X}(s(X,p);p)\in \tau$ for every $(X,p)\in \U_0\times V_0.$ Then, we  define the \textbf{full transition map} of $X\in\U_0$ at $p_{0}$ as the map
\begin{equation*}
\begin{array}{lcll}
\overline{T^{X}_{p_{0}}}:&(\s\cap V_0)_{p_{0}}&\longrightarrow& \tau\\
& p&\longmapsto& \p_{X}(s(X,p);p),
\end{array}
\end{equation*}
where $(\s\cap V_0)_{p_{0}}$ is the connected component of $\s\cap V_0$ containing $p_{0}$.

Throughout this paper, when $p_{0}$ and $p_{1}$ belong to the same orbit of $X,$ $\arcf{p_{0}}{p_{1}}{X}$ will denote the \textbf{oriented arc-orbit} of $X$ with extrema $p_{0}$ and $p_{1}$, i. e. $\arcf{p_0}{p_1}{X}=\p_{X}(I;p_{0})$ where $I=[0,t_1]$, $p_{0}=\p_{X}(0,p_0),$ and  $p_{1}=\p_{X}(t_{1},p_0).$ We shall omit the index $X$ if there is no ambiguity. Since we are constructing transition maps for nonsmooth vector fields, it is only considered orbits of $X$ which are contained in either $\overline{M^{+}}$ or $\overline{M^{-}}$. So, the domain of the full transition map has to be restricted to the following set
$$\sigma_X=\left\{ p\in(\s\cap V_0)_{p_{0}};\ \arc{p\overline{T^{X}_{p_{0}}}(p)}\ \text{ is contained in }\overline{M^{\pm}}\right\}.$$
Accordingly, the \textbf{transition map} of $X$ at $p_{0}$ is defined as
$T_{p_{0}}^{X}:=\overline{T_{p_{0}}^{X}}\big|_{\sigma_X}.$ 

It is worth to notice that $p_0$ may not be contained in the domain $\sigma_X$ of the transition map $T^X_{p_0}$ (see Figure \ref{p0fig}). However, if $X$ is defined in $\overline{M^{\pm}}$ and $q_{0}\in M^{\pm}$ (recall that $q_0$ defines the local transversal section $\tau$), then $p_{0}\in\sigma_X$ provided that the arc-orbit $\arc{p_{0}q_{0}}$ of $X$ is contained in $M^{\pm}$. 

\begin{figure}[h!]
	\centering
	\bigskip
	\begin{overpic}[width=9cm]{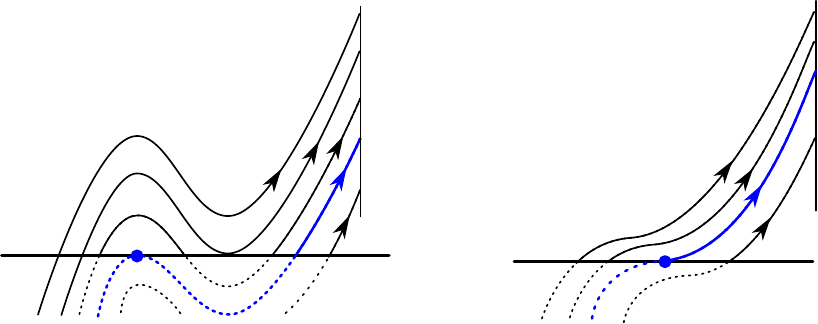}
	\put(45,37){{\footnotesize $\tau$}}\put(100.5,37){{\footnotesize $\tau$}}
	\put(45.5,4.5){{\footnotesize $\s$}}\put(98,4){{\footnotesize $\s$}}
	\put(15.5,5){{\footnotesize $p_0$}}\put(80.5,5){{\footnotesize $p_0$}}
	\put(45,21.5){{\footnotesize $q_0$}}\put(100.5,30.5){{\footnotesize $q_0$}}
	\end{overpic}
	\bigskip
	\caption{Unfolding of a transition map at a cusp point. Left: $p_0\notin\sigma_X$. Right: $p_0\in\sigma_X$.}\label{p0fig}	
\end{figure}

In Section \ref{charac_sec}, we characterize the full transition map $\overline{T^{X_0}_{p_0}}$ for vector fields $X_0\in\Xr$ having a $n$-multiplicity contact with $\s$ at $p_0$. Moreover, we describe how $\overline{T^X_{p_0}}$ behaves for $X$ in a small neighborhood of $X_0$ in $\Xr$.

\subsection{Mirror Maps}\label{involution_sec}

Assume that $X_0$ has a $2n$-multiplicity contact with $\s$ at $p_0$ for some $n\in\N$. We shall see that, for each $p\in\s$ near $p_0$, with $p\neq p_0$, there exists a time $t(p)$ such that $t(p_0)=0$ and $\p_{X_0}(t(p);p)\in\s$. Moreover, the flow of $X_0$ will define a germ of diffeomorphism at $p_0,$
\begin{equation*}\label{invol}
\begin{array}{lcll}
\rho:& (\s,p_0)&\longrightarrow&  (\s,p_0)\\
& p&\longmapsto& \p_{X_0}(t(p);p).
\end{array}
\end{equation*}
In this case, $\rho(p_0)=p_0$ and we say that $\rho$ is the \textbf{involution} associated with $X_0$ at $p_0$.

Through a local change of coordinates and a rescaling of time, we can assume that $p_0=(0,0)$, $h(x,y)=y$, and
\begin{equation}\label{expr2n}
X_0(x,y)=\left(\begin{array}{c}
1\\
\ell_0 x^{2n-1}+\er(x^{2n},y)
\end{array}\right),
\end{equation}
where $\ell_0>0$. In this case, for each $p\in\s$ the orbit connecting $p$ and $\rho(p)$ will be contained in $\overline{M^-}$ (see Figure \ref{invfig}).
\begin{figure}[h!]
	\centering
	\begin{overpic}[width=7cm]{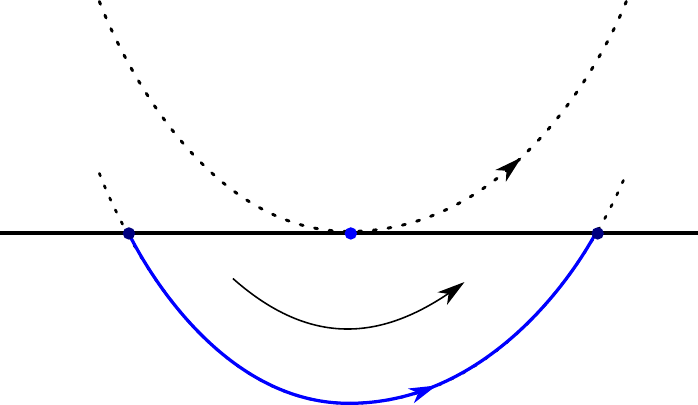}
			\put(99,20){{\footnotesize $\s$}}\put(47,20){{\footnotesize $p_0$}}		\put(17,20){{\footnotesize $p$}}		\put(85,20){{\footnotesize $\rho(p)$}}
	\end{overpic}
	\caption{Involution $\rho$ of $X_0$ at $p_0$.}\label{invfig}	
\end{figure}

Notice that $\p_{X_0}(t(x);(x,0))\in\s$ if, and only if, $\pi_2\circ \p_{X_0}(t(x);(x,0))=0$. In this case, $\rho(x)=x+t(x)$. Expanding $\p_{X_0}$ around $t=0$ we get
\begin{equation}\label{segundacoord}\pi_2\circ \p_{X_0}(t;(x,y))= y+\sum_{i=1}^{2n}\dfrac{X_0^{i}h(x,y)}{i!}t^i +\er(t^{2n+1}).\end{equation}

From \eqref{expr2n}, we  see that 
\begin{equation}\label{exprlie}X_0^i h(x,y)=\ell_0 \dfrac{(2n-1)!}{(2n-i)!}x^{2n-i}+\er(x^{2n-i+1},y).\end{equation}

Now, define the map
\begin{equation*}
S(s,x)=\dfrac{2n}{\ell_0 x^{2n}}\pi_2\circ\p_{X_0}(sx;(x,0)).
\end{equation*}
Notice that, if  $S(s,x)=0$, $x\neq 0$, and $s\neq 0$, then $\pi_2\circ \p_{X_0}(sx,(x,0))=0$. From \eqref{segundacoord} and \eqref{exprlie} we obtain that
$$S(s,x)=(1+s)^{2n}-1+\er(x).$$

Since $S(-2,0)=0$ and $\partial_s S(-2,0)=-2n>0$, it follows from the Implicit Function Theorem that there exists $s(x)=-2+\er(x)$ such that $S(s(x),x)=0$. From the definition of $S$, for $t(x)=xs(x)$, we have that $\p_{X_0}(t(x);(x,0))\in\s$, and then the involution $\rho$ is straightly defined. 

	From the construction above, it follows that there exists a compact neighborhood $V_0\subset M$ of $p_0$ such that the involution $\rho:(\s\cap V_0)_{p_0}\rightarrow (\s\cap V_0)_{p_0}$ is well defined and characterized as \begin{equation}\label{rho}
	\rho(x)=x+t(x)=-x+\er(x^2).
	\end{equation}
	
Now, we show that the a vector field $X\in\Xr$ sufficiently near $X_0$ still induces an involution in $(\s\cap V_0)_{p_0}$ but a finite set of points. In what follows  we also characterize it. For simplicity, identify $(\s\cap V_0)_{p_0}$ with the interval $[-\e_0,\e_0]$ and $p_0$ with $0$.

From definition of $\rho$, there exists $\e_0^*>0$ such that the intervals $I=[-\e_0,-\e_0/2]$ and $\rho(I)=[\e_0^*,\e_0]$ are connected by orbits of $X_0$ contained in $M^{-}$, and $X_0$ is transverse to $\s$ at every point of $I\cup\rho(I)$. Since $I$ is compact, given $\e>0$, there exists a small neighborhood $\U_1\subset\Xr$ of $X_0$ such that, for each $X\in\U_1$, there exist $\e_X^*,\e_X>0$ satisfying
\begin{enumerate}[i)]
\item $|\e_X^*-\e_0^*|,|\e_X-\e_0|<\e$;
\item each point of $I$ is connected to a unique point of $[\e_X^*,\e_X]$ through an orbit of $X$ contained in $M^-$;
\item $X$ is transverse to $\s$ at each point of $I\cup[\e_X^*,\e_X]$. 
\end{enumerate}

Notice that $[-\e_0/2,\e_X^*]$ and the orbit connecting $-\e_0/2$ and $\e_X^*$ give rise to a compact region $K^-$ of $M^-$ such that $X$ is regular at every point of $K^-$ (see Figure \ref{regKfig}). Thus, each orbit of $X$ entering in $K^-$ must leave it through another point. It allows us to see that $Xh$ has at least one zero in $(-\e_0/2,\e_X^*)$ and it has to be an even multiplicity contact of $X$ with $\s$ having the same concavity of $p_0$. Throughout this section, an even multiplicity contact of a vector field $X$ with $\s$ having the same concavity of $p_0$ will be called invisible, otherwise it will be called visible.

\begin{figure}[h!]
	\centering
	\begin{overpic}[width=7cm]{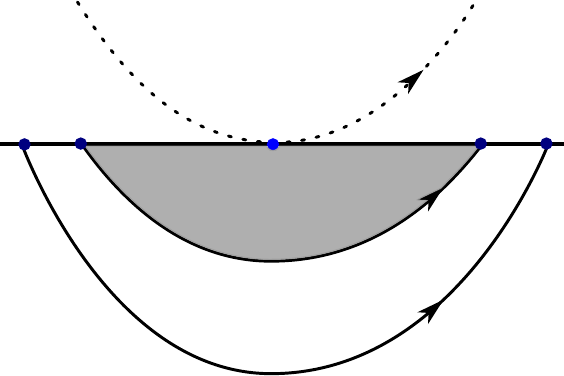}
		\put(99,37){{\footnotesize $\s$}}
		\put(47,42){{\footnotesize $0$}}
		\put(9.5,44){{\footnotesize $-\frac{\varepsilon_0}{2}$}}	\put(0,44){{\footnotesize $-\varepsilon_0$}}
		\put(97,44){{\footnotesize $\varepsilon_X$}}		\put(82,44){{\footnotesize $\varepsilon_X^*$}}	
		\put(47,30){{\footnotesize $K^-$}}	
	\end{overpic}
	\caption{Compact region $K^-$ for $X\in\mathcal{U}_1$.}\label{regKfig}	
\end{figure}

Since $X_0h(x)=\ell_0 x^{2n-1}+\er(x^{2n})$, there exist a neighborhood $\U_0\subset \U_1$ of $X_0$, $\Cr$ functions $a_i:\U_0\rightarrow (-\e,\e)$ such that $a_i(X_0)=0$, $i=0,\cdots, 2n-2$,  and a positive function $\ell:\U_0\rightarrow (\ell_0-\e,\ell_0+\e)$ with $\ell(X_0)=\ell_0$ satisfying
$$Xh(x)=P_X(x)+\er(x^{2n}),$$
where $P_X(x)=\sum_{i=0}^{2n-2}a_i(X)x^{i}+\ell(X)x^{2n-1}$. Furthermore, we can take the initial neighborhood $V_0$ sufficiently small such that the zeroes of $Xh$ in $[-\e_0,\e_0]$ are controlled by the polynomial $P_X$. Hence, it follows that there exist exactly $N_X$ points $r_i\in(-\e_0/2,\e_X^*)$, with $1\leq N_X\leq 2n-1$, such that $X$ has a $n_i$-multiplicity contact with $\s$ at $r_i$ for some $n_i\geq 2$, $i=1,\cdots, N_X$. In this case, $n_i\leq 2n$. Accordingly, let $\mathcal{E}_X$ be the finite subset of $(-\e_0/2,\e_X^*)$ containing
\begin{enumerate}[i)]
	\item $r_i$, $i=1,\cdots, N_X$, such that either $n_i$ is odd or $n_i$ is even and $X$ has a visible contact with $\s$ at $r_i$;
	\item $p\in (-\e_0/2,\e_X^*)$, such that $p$ and $r_i$ belong to the same orbit of $X$, for some $i=1,\cdots, N_X$, and the arc-orbit of $X$ with extrema $p$ and $r_i$ is contained in $\overline{M^-}$ (see Figure \ref{exfig}).
\end{enumerate}

\begin{figure}[h!]
	\centering
	\bigskip
	\begin{overpic}[width=7cm]{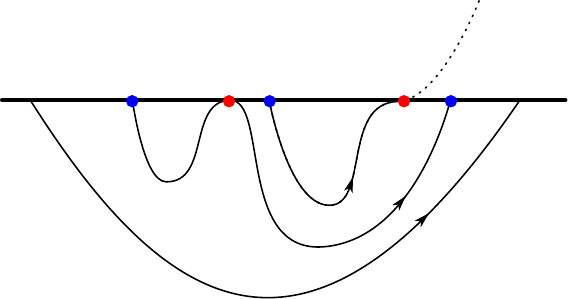}
		\put(96,30){{\footnotesize $\s$}}
		\put(22,37){{\footnotesize $p_1$}}
		\put(38,37){{\footnotesize $p_2$}}
		\put(45,37){{\footnotesize $p_3$}}
		\put(69,37){{\footnotesize $p_4$}}
		\put(78,37){{\footnotesize $p_5$}}
	\end{overpic}
	\bigskip
	\caption{Example of some points in $\mathcal{E}_X$, $p_1$, $p_3$, $p_5$ satisfy condition $(ii)$ and $p_2,p_4$ satisfy condition $(i)$.}\label{exfig}	
\end{figure}

If $r_i\in (-\e_0,\e_X)\setminus \mathcal{E}_X$, for some $i=1,\cdots, N_X$, then $X$ has an invisible even multiplicity contact with $\s$ at $r_i$. So, applying the same process above we find $\e_i^-,\e_i^+>0$ sufficiently small and an involution $\rho_X^i: (r_i-\e_i^-, r_i+\e_i^+)\rightarrow (r_i-\e_i^-, r_i+\e_i^+)$ induced by the flow of $X$ at $r_i$. In this case, $\rho_X^i$ is a diffeomorphism with a unique fixed point at $r_i$, and
\begin{equation}\label{rhoi}
\rho_X^i(x)=r_i-(x-r_i)+\er((x-r_i)^2).
\end{equation}

Now, if $p\in [-\e_0,\e_X]\setminus (\mathcal{E}_X\cup\{r_1,\cdots,r_{N_X}\})$, then $X$ is transverse to $\s$ at $p$ and there exists a unique point $p^*\in (-\e_0/2,\e_X^*)\setminus (\mathcal{E}_X\cup\{r_1,\cdots,r_{N_X}\})$ such that $X$ is transverse to $\s$ at $p^*$, $p$ and $p^*$ belong to the same orbit of $X$, and the arc-orbit of $X$ with extrema $p$ and $p^*$ is contained in $\overline{M^-}$. It allows us to extend the involutions $\rho_X^i$ to an involution $$\overline{\rho_X}: [-\e_0,\e_X]\setminus\mathcal{E}_X\rightarrow [-\e_0,\e_X]\setminus\mathcal{E}_X,$$
induced by the flow of $X$. We refer $\overline{\rho_X}$ as the \textbf{involution} of $X$ at $p_0$. 

Notice that $\overline{\rho_X}$ is a diffeomorphism for which $r_i$, $i=1,\cdots, N_X,$ are its only fixed points. Moreover, these points are invisible ever multiplicity contact of $X$ with $\s$ and the expansion of $\overline{\rho_X}$ at these points is given by \eqref{rhoi}. Thus $\overline{\rho_X}$ is completely characterized and $\overline{\rho_{X_0}}=\rho,$ where $\rho$ is given by \eqref{rho}.

We aim to use these involutions for detecting closed connections of nonsmooth vector fields. Thus, in order to avoid pseudo-connections (see \cite{GTS} for more details), we restrict $\overline{\rho_X}$ to the set  
$$\sigma_X^{inv}=\left\{ p\in[-\e_0,\e_X]\setminus\mathcal{E}_X;\ Xh(p)\leq 0\right\}.$$
Accordingly, the restriction $\rho_{X}:=\overline{\rho_X}\big|_{\sigma_X^{inv}}$ is referred as  \textbf{mirror map} of $X$ at $p_{0}$. The condition $Xh(p)\leq 0$ on the domain $\sigma_X^{inv}$ comes from the initial assumptions which imply that the orbit connecting $p$ and $\rho_X(p)$ is contained in $\overline{M^-}$ for every $p\in\sigma_X^{inv}$. When considering nonsmooth systems these orbits could be contained in $\overline{M^+}.$ In this case, the condition on $\sigma_X^{inv}$ is changed to $Xh(p)\geq 0.$

\subsection{Displacement Functions}\label{disp_sec}

Now, we are able to define the displacement functions associated with a $\s$-polycycle $\Gamma_0$ of $Z_0=(X_0,Y_0)$. Assume that $\Gamma_0$ has $k$ tangential singularities $p_{i}$ of multiplicity $n_{i}\in\N$, $1\leq i\leq k$. Let $\gamma_{i}$ be the regular orbit of $Z_{0}$ connecting $p_{i}$ to $p_{i+1}$, $i=1,\ldots, k-1$,  $\gamma_{k}$ be the regular orbit of $Z_{0}$  connecting $p_{k}$ and $p_{1}$, and consider sufficiently small neighborhoods $U_i$ of $p_i$, $1\leq i\leq k$.  Notice that for each $p_i$, $i\in\{1,\ldots, k\},$ one of the following statements hold:

\begin{enumerate}
	\item[(\textbf{E})] $\Gamma_0\cap U_i\setminus\{p_{i}\}$ is contained in either $M^+$ or $M^-$;
\item[(\textbf{O})]	$\Gamma_0\cap U_i\setminus\{p_{i}\}$ has one connected component in $M^+$ and the other one in $M^-$. 
\end{enumerate}

Suppose that {\bf (O)} holds for $p_{i}$ and assume, without loss of generality, that $W^{u}(p_i)\cap\Gamma_0\cap U_i\subset M^+$ and $W^{s}(p_i)\cap\Gamma_0\cap U_i\subset M^-$. Let $\tau_{i}^{u}$ and $\tau_{i}^{s}$ be transversal sections of $X_0$ and $Y_0$ at the points $q_{i}^{u}\in W_{+}^{u}(p_i)$ and $q_{i}^{s}\in W_{-}^{s}(p_i)$, which are contained in $U_{i}$, respectively. From the construction performed in Section \ref{trans-sec}  there exist  transition maps of $X_0$ and $Y_0$ at $p_0,$ $T^u_{i}:\sigma_i(X_0)\rightarrow \tau_i^u$ and $T^s_{i}:\sigma_i(Y_0)\rightarrow \tau_i^s$, respectively.

Now, suppose that {\bf (E)} holds for $p_{i}$ and assume, without loss of generality, that $\Gamma_0\cap U_i\subset \overline{M^+}$. Let $\tau_{i}^{s}$ and $\tau_{i}^{u}$ be transversal sections of $X_0$ at the points $q_{i}^{s}\in W_{+}^{s}(p_i)$ and $q_{i}^{u}\in W_{+}^{u}(p_i)$, which are contained in $U_{i}$, respectively. In this case, we have two distinguished situations:

\textbf{(I)} If $\s\cap U_i\setminus{\{p_i\}}$ has one connected component in the sliding region of $Z_0$, then let $\sigma_i^\pitchfork(X_0)$ be the restriction to $\overline{M^+}$ of a local transversal section of $X_0$ at $p_i$. Clearly, the flow of $X_0$ induces maps $T^u_{i}:\sigma_i^\pitchfork(X_0)\rightarrow \tau_i^u$ and $T^s_{i}:\sigma_i^\pitchfork(X_0)\rightarrow \tau_i^s$, which are restrictions of diffeomorphisms.

	\textbf{(II)} If $\s\cap U_i \setminus{\{p_i\}}\subset \Sigma^c$, then besides the maps $T^u_{i}:\sigma_i^\pitchfork(X_0)\rightarrow \tau_i^u$ and $T^s_{i}:\sigma_i^\pitchfork(X_0)\rightarrow \tau_i^s,$ induced by the flow of $X_0$, we can also define other maps in the following way: first, notice that this situation is only possible when $Y_0$ has an invisible even multiplicity contact with $\s$ at $p_i$, and thus, we consider the mirror map $\rho_i:\sigma_i^{inv}(Y_0)\rightarrow \s\cap U_i$ of $Y_0$ at $p_i$ (see Section \ref{involution_sec}). Now, let $T^{X_0}_{-}: \sigma_i^-(X_0)\rightarrow \tau_i^s $ and $T^{X_0}_{+}: \sigma_i^+(X_0)\rightarrow \tau_i^u$ be the transition maps of $X_0$ at $p_0$ with respect to the transversal sections $\tau_i^s$ and $\tau_i^u$, respectively. Now, define the section
	\begin{equation*}\sigma_i^t(Z_0)=\rho_i^{-1}(\sigma_i^+(X_0)\cap \rho_i(\sigma_i^{inv}(Y_0))),\end{equation*}
and the maps
	\begin{equation}\label{Tsui}
	\begin{array}{l}
	T^s_i:\sigma_i^-(X_0)\rightarrow \tau_i^s,\quad T^s_i=T^{X_0}_{-},\vspace{0.2cm}\\
	T^u_i:\sigma_i^t(Z_0)\rightarrow \tau_i^u,\quad T^u_i=T^{X_0}_{+}\circ\rho_i.
	\end{array}
	\end{equation}
 Thus, in this case, we have maps $T_i^{u,s}$ induced by crossing orbits of $Z_0$.

Summarizing, if $p_i$ has type \textbf{(O)}, \textbf{(E-I)} or \textbf{(E-II)}, then we define $\sigma_i(Z_0)$ as $\sigma_i(X_0)\cap \sigma_i(Y_0)$, $\sigma_i^{\pitchfork}(X_0)$ or $\sigma_i^{\pitchfork}(X_0)\cup (\sigma_i^t(Z_0)\cap \sigma_i^-(X_0))$, respectively. So, in any case, we construct maps $T_i^{u,s}:\sigma_i(Z_0)\rightarrow \tau_i^{u,s}$ induced by crossing orbits of $Z_0$. We refer the maps $T_i^{u,s}$ as \textbf{transfer functions} (see Figure \ref{transferfig}).

\begin{figure}[h!]
	\centering
	\bigskip
	\begin{overpic}[width=12cm]{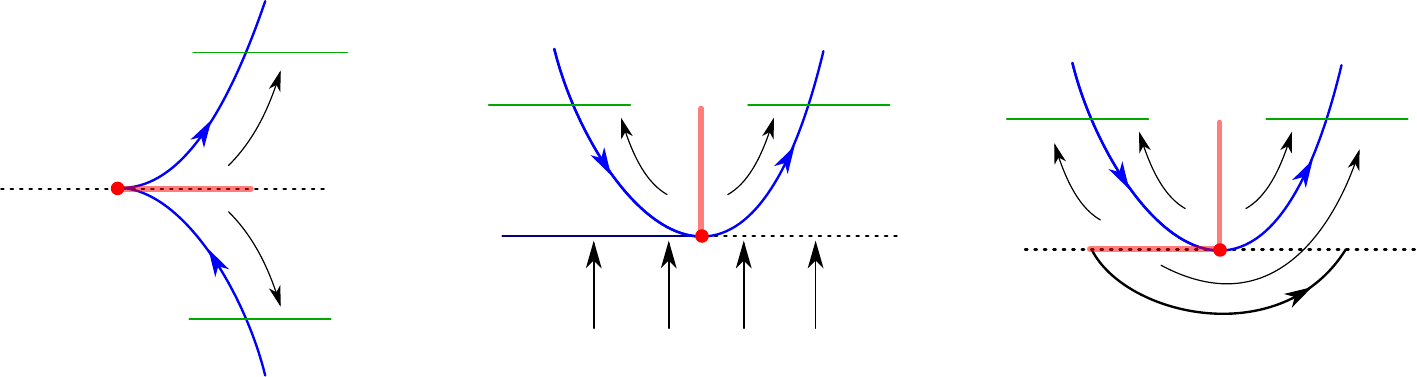}
		\put(0,11){{\footnotesize $\s$}}
		\put(7,11){{\footnotesize $p_i$}}
		\put(20,17){{\footnotesize $T_i^u$}}
		\put(20,9){{\footnotesize $T_i^s$}}		
		\put(25,23){{\footnotesize $\tau_i^u$}}
\put(24,4){{\footnotesize $\tau_i^s$}}		
\put(15,-2){{\footnotesize $(a)$}}	
\put(49,-2){{\footnotesize $(b)$}}	
\put(85,-2){{\footnotesize $(c)$}}		
		\put(46,15){{\footnotesize $T_i^s$}}	
				\put(50.5,15){{\footnotesize $T_i^u$}}	
		\put(63,20){{\footnotesize $\tau_i^u$}}
\put(32.5,20){{\footnotesize $\tau_i^s$}}		
		\put(49,8){{\footnotesize $p_i$}}	
			\put(85.5,7){{\footnotesize $p_i$}}	
					\put(100,19){{\footnotesize $\tau_i^u$}}
			\put(69,19){{\footnotesize $\tau_i^s$}}				
		\put(72,12){{\footnotesize $T_i^s$}}	
\put(96,12){{\footnotesize $T_i^u$}}						
	\end{overpic}
	\bigskip
	\caption{Transfer functions of types (a)-\textbf{(O)}, (b)-\textbf{(E-I)} and (c)-\textbf{(E-II)}.}\label{transferfig}	
\end{figure}

Now, the regular orbit $\gamma_i$ connecting $p_{i}$ to $p_{i+1}$, $i=1,\cdots,k$, induces a diffeomorphism $D_{i}: \tau_{i}^{u}\rightarrow \tau_{i+1}^{s}$ such that $D_i(p_i)=p_{i+1}$.

For a sufficiently small neighborhood $\V\subset \Or$ of $Z_0$ in $\Or$, we  see that all the maps used to construct the transfer functions  $T_i^{u,s}$ above are also defined for each $Z\in\V$ (see Sections \ref{trans-sec} and \ref{involution_sec} ). Thus, for each $Z\in\V,$ the transfer functions $T_i^{u,s}(Z): \sigma_i(Z)\rightarrow \tau_i^{u,s}$ and the diffeomorphisms $D_{i}(Z): \tau_{i}^{u}\rightarrow \tau_{i+1}^{s}$ can be constructed in the same way as described above. In particular, the domain $\sigma_i^{\pitchfork}(X_0)$ is perturbed into
\[\sigma_i^{\pitchfork}(X)=\{p\in\sigma_i^{\pitchfork}(X_0);\ \arcf{p}{T^u_i(Z)(p)}{X} \textrm{ and }\arcf{T_i^s(Z)(p)}{p}{X} \textrm{ are contained in } \overline{M^+}\}.\]

We now relate all these information through displacement functions. 

\begin{definition}\label{disp_def}
	The \textbf{$i$-th displacement function} of $Z$ is defined as
	\begin{equation*}
	\begin{array}{lcll}
	\Delta_{i}(Z):&\sigma_{i}(Z)\times\sigma_{i+1}(Z)&\longrightarrow& \R\\
	& (x_{i},x_{i+1})&\longmapsto& \phi\circ T_{i}^{u}(Z)(x_{i})- \phi\circ D_{i}^{-1}\circ T_{i+1}^{s}(Z)(x_{i+1}),
	\end{array}
	\end{equation*}	
	where $\phi:\tau_{i}^{u}\rightarrow \R$ is a parameterization of $\tau_{i}^{u}$.
\end{definition}

	Clearly, the zeroes of the $i-$th displacement function of $Z$ does not depend on the parameterization of $\tau_{i}$. It is straightforward to see that two points, $x_{i}\in\sigma_{i}(Z)$ and $x_{i+1}\in\sigma_{i+1}(Z)$, are connected through an orbit of $Z$ if, and only if, $\Delta_{i}(Z)(x_{i},x_{i+1})=0$.

\begin{figure}[h!]
	\centering
	\bigskip
	\begin{overpic}[width=10cm]{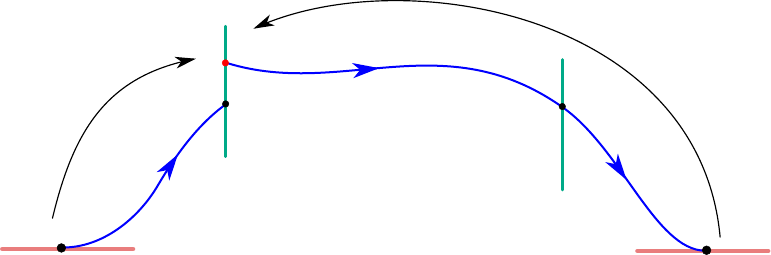}
	
		\put(8,-2){{\footnotesize $p_i$}}
		\put(91,-2){{\footnotesize $p_{i+1}$}}
		\put(5,20){{\footnotesize $T_i^u(Z)$}}	
			\put(87,20){{\footnotesize $D_i^{-1}\circ T_i^s(Z)$}}	
		\put(30,14){{\footnotesize $\tau_i^u$}}	
	\put(69,10){{\footnotesize $\tau_i^s$}}	
		\put(30,21.5){{\tiny $\Delta_i(Z)$}}		
	\end{overpic}
	\bigskip
	\caption{Construction of the $i-$th displacement function of $Z\in\V$.   }\label{dispfig}	
\end{figure}

\begin{remark}
		We emphasize that the construction of displacement functions as in Definition \eqref{disp_def} allows us to describe the complete bifurcation diagrams of a vector field in $\Or$ around many different types of $\s$-polycycles, in particular the ones analyzed later on in this paper. We highlight that in all the cases all the bifurcating crossing limit cycles with the same topological type of $\Gamma_0$ are detected by this method. However, there exist tangential singularities which admit bifurcation of global connections in their local unfoldings, for instance the cusp-cusp singularity. In these cases, such global connections would not be detected by our method for $\s$-polycycles through these singularities.	
		\end{remark}

\section{Characterization of Transition Maps}\label{charac_sec}

In this section we characterize the transition maps of $Z_0=(X_0,Y_0)$ at $p\in\s$ and we also study how they typically change for unfoldings of $Z_0$. 

Firstly, notice that if $X\in\Xr$ is transversal to $\s=h^{-1}(0)$ at $p$, then the transition map $T^X_{p}|_{\sigma}$ is a diffeomorphism at $p$ and $\sigma$ is an open set of $\s$ containing $p$. 

Now, assume that $X\in\Xr$ has a $n-$multiplicity contact  with $\s$ at $p$. Consider coordinates $(x,y)$ at $p$ (i.e. $x(p)=y(p)=0$) such that $h(x,y)=y$ and write $X=(X_{1},X_{2})$ in this coordinate system. 
In this case $X_{1}(0,0)\neq 0$, and thus $X_{1}(x,y)\neq 0$, for every $(x,y)$ in some neighborhood $U$ of the origin. By performing a time rescaling, 
 we obtain that $X(x,y)$ and $\widetilde{X}(x,y)=(\sgn(X_{1}(0,0)), f(x,y))$, with $f(x,y)=X_{2}(x,y)/|X_{1}(x,y)|$, have the same integral curves in $U$.  It is easy to see that $Xh(x,y)=|X_{1}(x,y)|\widetilde{X}h(x,y).$
In general, $X^{i}h(0,0)=0$ if, and only if, $\widetilde{X}^{i}h(0,0)=0$. Moreover, one can prove that $X^{i}h(0,0)$ and $\widetilde{X}^{i}h(0,0)$ have the same sign. In what follows, without loss of generality, we take $X(x,y)=(\dg, f(x,y))$, with $\dg=\pm 1$.

\begin{lemma}\label{lemma_Xh}
	Assume that $X=(\dg,f(x,y))$, with $\dg=\pm1$, has a $n$-multiplicity contact with $\s$ at $(0,0)$, i.e. $X^{i}h(p)=0$, $i=0,1,\ldots,n-1$, and $X^{n}h(p)\neq 0$. Then:
	\begin{itemize}
	\item[(a)]
	$
	\dfrac{\partial^{i-1}f}{\partial x^{i-1}}(0,0)=0,\,\, \text{for}\,\, i=1,2,\ldots, n-1,\,\, \text{and}\,\, \dfrac{\partial^{n-1}f}{\partial x^{n-1}}(0,0)\neq 0.
	$
	\bigskip
\item [(b)]$Xh(x,0)=\ag x^{n-1}+\er(x^{n}),$ 
	where $\sgn(\ag)=\dg^{n-1}\sgn(X^{n}h(0,0))$.
	\end{itemize}
\end{lemma}
\begin{proof}
	Firstly,  the statement (a) follows by noticing that $0=Xh(0,0)=f(0,0)$ and 
	$$X^{i}h(0,0)=\dg^{i-1} \dfrac{\partial^{i-1} f}{\partial x^{i-1}}(0,0).$$
	Now, since $Xh(x,0)=\langle X(x,0),(0,1)\rangle=f(x,0)$,  expanding $Xh(x,0)$ in Taylor series around $x=0$, we obtain that
	$$Xh(x,0)= \dfrac{\partial^{n-1} f}{\partial x^{n-1}}(0,0)x^{n-1}+\er(x^{n}).$$
	Hence, the statement (b) follows by taking  $\ag=\dg^{n-1}X^{n}h(0,0)$.	 
\end{proof}

From Lemma \ref{lemma_Xh} it follows that $X$ is transversal to $\s$ for every $(x,0)\in V\cap \s\setminus\{(0,0)\}$, where $V$ is a small neighborhood of the origin $O$.
Let  $X$ be defined in $\overline{M^{\pm}}$ and assume that 
\begin{enumerate}\item[\textbf{(A)}]either the oriented arc-orbit $\arc{O q_0}|_X$ or $\arc{q_0O}|_X$  is contained in $M^{\pm}$.\end{enumerate} In the first case $q_0=\p_{X}(T_0;0,0)\neq(0,0)$, and in the second one $q_0=\p_{X}(-T_0;0,0)\neq(0,0)$, for some $T_0>0$. 

Let $q_0=(x_0,y_0)$. Since $\pi_{1}(X(q_0))=\dg\neq 0$, it follows that \begin{equation}\label{tau}\tau=\{(x_{0},y);\ y\in (y_{0}-\e,y_{0}+\e)\}\end{equation} is a transversal section of $X$ at $q_{0}$, for $\e$ sufficiently small. Take $\e>0$ such that $\tau \subset V\cap   M^{\pm}$. Therefore, the full transition map of $X$ at $(0,0)$ is  $T_{X}: (V\cap\s)_{0} \rightarrow \tau$ given by
\begin{equation*}
T_{X}(x,0)=(x_{0},\pi_{2}(\p_{X}(\dg x_{0} - \dg x;x,0))).
\end{equation*}
Now, we use Lemma \ref{lemma_Xh} to determine the domain $\sigma$ of the transition map of $X$ at $(0,0)$.

\begin{corollary}\label{charac_sigma}
	Assume that $X$ has a $n$-multiplicity contact with $\s$ at $(0,0)$. Then, the following statements hold:
	\begin{enumerate}[i)]
		\item
		if $n$ is odd, then $\sigma=(-\e,\e)\times\{0\}$, for $\e>0$ sufficiently small;
		\item
		 if $n$ is even, then $\sigma=I\times\{0\}$, where $I$ is either $[0, \e)$ or $(- \e,0]$, for $\e>0$ sufficiently small.	\end{enumerate}
\end{corollary}
\begin{proof}
	If $n$  odd, then $Xh(x,0)=\ag x^{k}+\er(x^{k+1})$, where $k=n-1$ is even. It means that  $\sgn(\ag)Xh(x,0)>0$ for $x\in (-\e,\e)\setminus\{0\}$ and $\e>0$ sufficiently small. So all the orbits of $X$ passing through $(-\e,\e)\times\{0\}$ enter (or leave) $M^{\pm}$.	If  $n$ is even, then $Xh(x,0)=\ag x^{k}+\er(x^{k+1})$, where $k$ is odd. It means that $\sgn(\ag)Xh(x,0)>0$, for $x\in (0,\e)$,  and $\sgn(\ag)Xh(x,0)<0$, for $x\in (-\e,0)$, where $\e>0$ is sufficiently small. We conclude the proof by observing that the transition map is defined in the unique domain where $Xh(x,0)$ has the same sign of $X_{2}(q_{0})$.
\end{proof}

In what follows we describe the expression of the full transition map $T_{X}$ of $X$ at $(0,0)$, when the origin is a  $n$-multiplicity contact.

\begin{theorem}\label{charac1}
	Suppose that $X\in\Xr$ has a $n$-multiplicity contact with $\s$ at $p=(0,0)$. In addition, assume that $X$ satisfies condition {\bf (A)}. Then the full transition map $T_{X}: (V\cap\s)_{0}\rightarrow \tau$ (where $\tau$ is given in \ref{tau}) is given by:
	\begin{equation*}
	T_{X}(x,0)= (x_{0},y_{0}+\kappa x^{n}+\er(x^{n+1})),
	\end{equation*}
	where $\sgn(\kappa)=-\dg^{n}\sgn(X^{n}h(0,0))$. 
\end{theorem}

\begin{proof}
	As we have seen before, we can assume that $X=(\dg, f(x,y))$. 
	Consider the change of coordinates $\phi(u,v)=(x(u,v),y(u,v)),$ where $x(u,v)=\dg u$ and $y(u,v)=\p_{X}^{2}(u;0,v)$ ($\p_{X}^{2}$ denotes $\pi_{2}\circ \p_{X}$).  Notice that 
	\begin{equation} \label{partialprop1}
	\begin{array}{l}
	\dfrac{\partial x}{\partial u}(0,0)=\dg,\, \dfrac{\partial x}{\partial v}(0,0)=0,\, \dfrac{\partial y}{\partial u}(0,0)=f(0,0)=0,\,\text{and}\vspace{0.2cm} \\
	\dfrac{\partial y}{\partial v}(0,0)=\left.\dfrac{\partial \p_{X}^{2}}{\partial y}(0;0,0)=1.\right.
	\end{array}
	\end{equation}
Therefore, $\phi$ is a diffeomorphism around the origin. In addition, it can be proved that $\phi$ is a conjugation between $X$ and $\mathcal{S}(u,v)=(1,0)$ (see \cite{H}). In this new coordinate system, $(u,v)$, $\Sigma$ and $\tau$ becomes, respectively,
	$$\widetilde{\s}=\phi^{-1}(\s)=\{(u,v)\in\rn{2};\p_{X}^{2}(u;0,v)=0\} \textrm{ and } \widetilde{\tau}=\{(\dg x_{0},v);\ v\in (-\e,\e)\}.$$ See Figure \ref{changeofcoordinates}.
	
	Since $\p_{X}^{2}(0;0,0)=0$ and from \eqref{partialprop1}, the Implicit Function Theorem implies the existence of $\gamma:(-\eta,\eta)\rightarrow \rn{}$ such that $\gamma(0)=0$ and
	$\widetilde{\s}=\{(u,\gamma(u))\in\rn{2};\ u\in (-\eta,\eta)\}.$
	
	Notice that $\p_{\mathcal{S}}(t;u,v)=(t+u,v)$, so the full transition map $T_{\mathcal{S}}:\widetilde{\s}\rightarrow \widetilde{\tau}$ is given by
	$$T_{\mathcal{S}}(u,\gamma(u))=\p_{\mathcal{S}}(\dg x_{0}-u, u, \cg(u) )=(\dg x_{0},\gamma(u)).$$
	
	Now, we must characterize the function $\gamma$ around $u=0$. Computing the $k$-th derivative of $\p_{X}^{2}(u;0,\gamma(u))=0$ in the variable $u$, and using that $\p_{X}(u;0,\gamma(u))=(\dg u,\p_{X}^{2}(u,0,\gamma(u)))=(\dg u,0)$, we get
	\begin{equation}\label{eq_deriv}
\gamma^{(k)}(u)=-\dg^{k-1}\dfrac{\partial^{k-1} f}{\partial x^{k-1}}(\dg u,0)\left(\dfrac{\partial \p_{X}^{2}}{\partial y}(u;0,\gamma(u))\right)^{-1}+  \sum_{i=1}^{k-1}P_{i}^{k}(u)\gamma^{(i)}(u),
	\end{equation}
where $P_{i}^{j}$ are continuous functions. 
	From Lemma \ref{lemma_Xh} (a) and equation \eqref{eq_deriv} we obtain that $\gamma^{(k)}(0)=0$, for every $1\leq k\leq n-1$ and
	$$\gamma^{(n)}(0)= - \dg^{n-1}\dfrac{\partial^{n-1} f}{\partial x^{n-1}}(0,0)=-X^{n}h(0,0).$$	
	Consequently, 
	$T_{\mathcal{S}}(u,\gamma(u))=(\dg x_{0}, \alpha u^{n}+\er(|u|^{n+1})),$
	where $\alpha=-X^{n}h(0,0)$.
	
	From the above construction, the following diagram is commutative.
	
	\begin{center}
		\begin{tikzpicture}
		\matrix (m) [matrix of math nodes,row sep=3em,column sep=4em,minimum width=2em]
		{
			\s & \tau \\
			\widetilde{\s} & \widetilde{\tau} \\};
		\path[-stealth]
		(m-1-1) edge node [left] {$\phi^{-1}$} (m-2-1)
		edge  node [below] {$T_{X}$} (m-1-2)
		(m-2-1.east|-m-2-2) edge node [below] {$T_{\mathcal{S}}$} (m-2-2)
		(m-1-2) edge node [right] {$\phi^{-1}$} (m-2-2);
		\end{tikzpicture}
	\end{center}

Since  $\pi_{1}\circ\phi^{-1}(x,0)=\dg x$ and $\phi^{-1}(x,0)\in \widetilde{\s}$, it follows that $\phi^{-1}(x,0)=(\dg x,\gamma( \dg x))$. Also, observe that $(x_{0},y_{0})=\p_{X}(T_{0},0,0)=(\dg T_{0}, \p_{X}^{2}(T_{0},0,0))$. So, $\dg x_{0}=T_{0}$. Hence,
	$$
	T_{X}(x,0)=\phi\circ T_{\mathcal{S}}\circ\phi^{-1}(x,0)=(x_{0},y_{0}+ \kappa x^{n}+\er(x^{n+1})),											
	$$
	where
	$$\kappa=-\dfrac{\partial \p_{X}^{2}}{\partial y}(T_0,0,0)X^{n}h(0,0)\dg^{n}.$$
	
	Finally , we can take $|T_{0}|$ small enough such that $\dfrac{\partial \p_{X}^{2}}{\partial y}(T_{0};0,0)>0$ since $\dfrac{\partial \p_{X}^{2}}{\partial y}(0;0,0)=1>0$. Therefore, $\sgn(\kappa)=-\dg^{n}\sgn(X^{n}h(0,0))$.
\end{proof}

\begin{figure}[h!]
	\centering
	\bigskip
	\begin{overpic}[width=10cm]{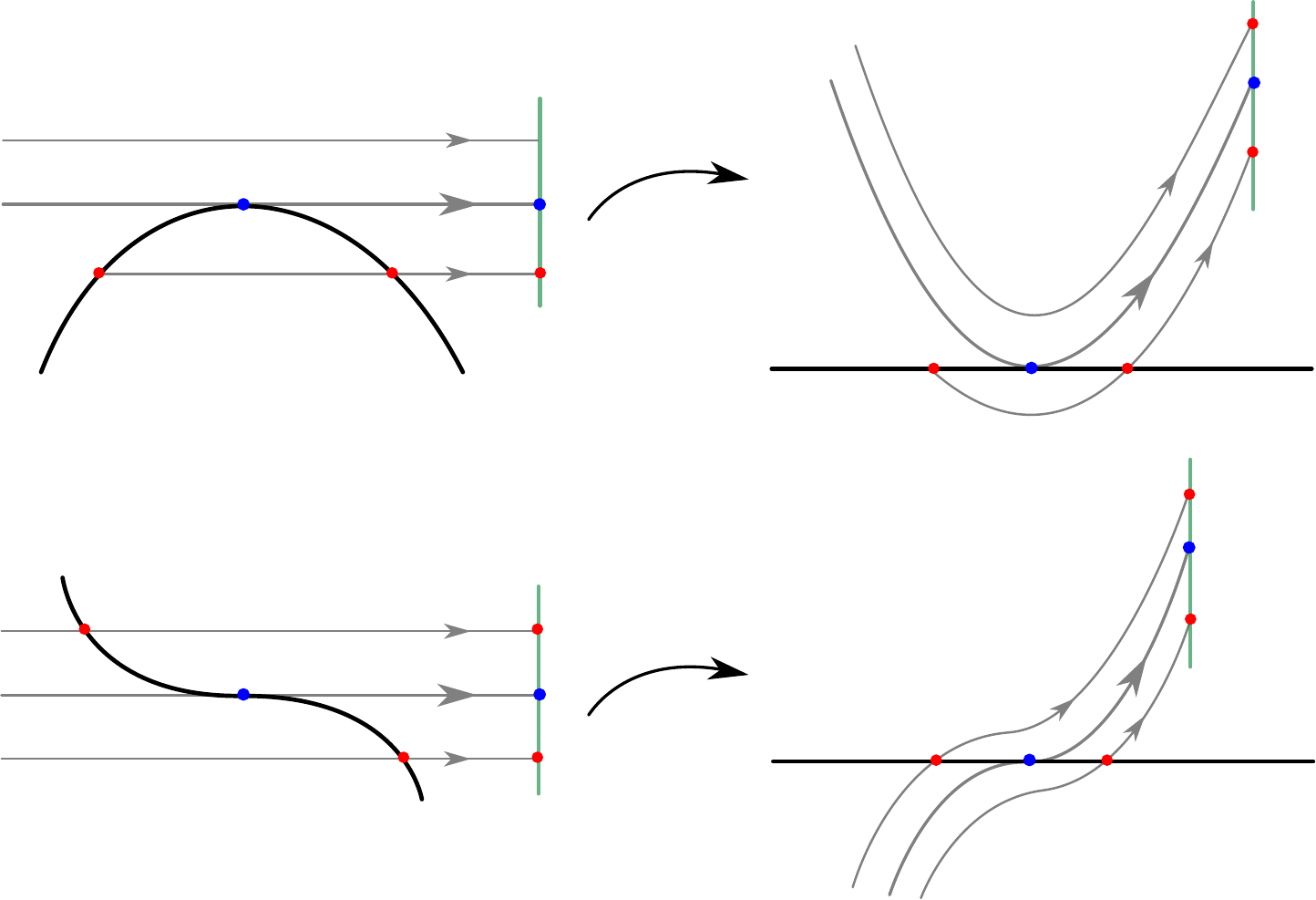}
		\put(32.5,6){{\footnotesize $\widetilde{\s}$}}
		\put(42,8){{\footnotesize $\widetilde{\tau}$}}
		\put(98,7){{\footnotesize $\s$}}
\put(91.5,18){{\footnotesize $\tau$}}	
		\put(36,40){{\footnotesize $\widetilde{\s}$}}
\put(42,45){{\footnotesize $\widetilde{\tau}$}}
\put(98,37){{\footnotesize $\s$}}
\put(96,53){{\footnotesize $\tau$}}			
	\end{overpic}
	\bigskip
	\caption{  Illustration of the change of coordinates $\phi^{-1}(u,v)$ at a fold and a cusp point. }	\label{changeofcoordinates}
\end{figure} 

Now, let $X_{0}\in\Xr$ satisfy the assumptions of Theorem \ref{charac1}.  We know that there exist $\e>0$ and a neighborhood $\mathcal{U}$ of $X_0$ such that a full transition map $T_X: (-\e,\e)\rightarrow (y_{0}-\e,y_{0}+\e)$ is defined for  each $X\in\mathcal{U}$ (see Section \ref{trans-sec}). In what follows we shall characterize this map.

\begin{theorem}\label{charac2}
	Suppose that $X_0\in\Xr$ has a $n$-multiplicity contact with $\s$ at $p=(0,0)$, with $n\geq2$. In addition, assume that $X_0$ satisfies condition {\bf (A)}. 
	Then, there exist a neighborhood $\U_0$ of $X_{0}$ in $\Xr$, $n-2$ surjective functions  $\lambda_{i}:\U_0\rightarrow(-\delta,\delta),$ $i=1,\cdots, n-2,$ depending continuously on $X$,  such that for each $X\in\U_0$ there exists a diffeomorphism $h_{X}: (-\e,\e)\rightarrow (-\e,\e)\times\{0\}$ for which the full transition map $T_{X}: (-\e,\e)\times\{0\}\rightarrow \tau$ is given by:
	\begin{equation*}
	T_{X}(h_{X}(x))=\Big(x_0, \lambda_{0}(X)+\kappa(X) x^{n}+ \sum_{i=1}^{n-2}\lambda_{i}(X)x^{i}+\er(x^{n+1})\Big),
	\end{equation*}
	where $\lambda_{0}=\pi_2\circ T_{X}(0,0)$, $\sgn(\kappa)=-\dg^{n}\sgn(X^{n}h(0,0))$ and $\dg=\pm1$. 
\end{theorem}

\begin{proof}
 In what follows, for the sake of simplicity, we shall identify $(-\e,\e)\times \{0\}$ and $\tau$ with the intervals $(-\e,\e)$ and $(y_0-\e,y_0+\e)$, respectively.	

From the discussion above, define the continuous map
	
	\begin{equation*}
	\begin{array}{lcll}
	T:&\U&\longrightarrow& \mathcal{C}_{0}^{\infty}(\R,\R)/\sim\\
	& X&\longmapsto& [T_{X} - T_{X}(0)],
	\end{array}
	\end{equation*}
	where $ \mathcal{C}_{0}^{\infty}(\R,\R)/\sim$ is the space of germs of $\mathcal{C}^{\infty}$ functions $f:\R\rightarrow\R$  such that $f(0)=0$, with the equivalence relation
	$$f\sim g \textrm{ if, and only if, } f-g=\er(x^{n+1}).$$
	As usual, $[f]$ denotes the equivalence class of $ \mathcal{C}_{0}^{\infty}(\R,\R)/\sim$ which contains $f\in\mathcal{C}_{0}^{\infty}(\R,\R)$.
	
	Denote $T(X_{0})$ by $T_0$ and notice that $T$ is surjective onto an open neighborhood of $T_0$ in $ \mathcal{C}_{0}^{\infty}(\R,\R)/\sim$. In fact, consider the vector field $X_{0}$ in the straightened form $\mathcal{S}=(1,0)$, then $\s$ is the graph $\{(x,h(x));\ x\in (-\e,\e)\}$ in these coordinates, for some $\e>0$ sufficiently small, and $T_{0}(x)=h(x)$ (see proof of Theorem \ref{charac1}). Therefore, any sufficiently small  perturbation of $h$ in the space of functions  corresponds to the transition map of a vector field $X$ in $\U$ by considering a small change in the coordinate system.
	
	From Theorem \ref{charac1} it follows that $T_0=[f_0]$, where $f_0(x)=\kappa x^n$. Now, since the stable unfolding of $f_0$ is given by $F_{\lambda}(x)=\kappa x^{n}+\sum_{i=1}^{n-2}\lambda_{i}x^{i}$, there exists a neighborhood $\mathcal{W}$ of $T_0$ in $\mathcal{C}_{0}^{\infty}(\R,\R)/\sim$ such that, for each $f\in\mathcal{W}$, there exist $n-2$ parameters $\lambda_{i}=\lambda_{i}(f)$ and a diffeomorphism $h_{f}:\R\rightarrow\R$, such that
	$$f(h_{f}(x))=\kappa x^{n}+ \sum_{i=1}^{n-2}\lambda_{i}x^{i}+\er(x^{n+1}).$$
	In addition, the parameters $\lambda_i$ and $h_{f}$ depend continuously on $f$.
	
	Taking $\U_0=T^{-1}(\mathcal{W})$, we have that for each $X\in\U_0$
	$$	T_{X}(h_{X}(x))= \lambda_{0}+\kappa x^{n}+ \sum_{i=1}^{n-2}\lambda_{i}x^{i}+\er(x^{n+1}),$$
	where $\lambda_{i}:\U_0\rightarrow(-\delta,\delta),$ for $i=1,\cdots, n-2,$ are surjective functions depending continuously on $X$ and $\lambda_{0}=\pi_2\circ T_{X}(0)$.
\end{proof}

\section{Regular-Tangential $\s$-Polycycles}\label{reg-tan_sec}

This section is devoted to apply the method of displacement functions, described in Section \ref{approach_sec}, for obtaining bifurcation diagrams of nonsmooth vector fields around some  regular-tangential $\s$-polycycles (see Definition \ref{def_scyclegeneralized}).  More specifically, in Section \ref{descripregtan}, we describe the displacement functions appearing in the crossing system \eqref{cross} for such $\s$-polycycles. In Section \ref{polyn_sec}, we prove that at most one crossing limit cycle bifurcates from $\s$-polycycles having  a unique regular-tangential singularity. Then, in Section \ref{polyn+_sec} we generalize the previous result for $\s$-polycycles having several regular-tangential singularities. In particular, the bifurcation diagrams of $\s$-polycycles having either a unique $\s$-singularity of regular-cusp type or only two singularities of  regular-fold type are completely described in Sections \ref{cusp_sec} and \ref{2fold-fold}, respectively.

\subsection{Description of the crossing system}\label{descripregtan}

\

Assume that $Z_{0}=(X_{0},Y_{0})\in\Or$ has a $\s$-polycycle $\Gamma_0$ containing $k$ regular-tangential singularities $p_{i}$ of multiplicity $n_{i}\in\N$, $1\leq i\leq k$. Consider a coordinate system $(x,y)$ satisfying that, for each $i\in\{1,2,\ldots,k\}$, $x(p_{i})=a_{i}$, $y(p_{i})=0$, and $h(x,y)=y$ near $p_{i}$.

Firstly, we shall characterize $\Gamma_0$ locally around each point $p_i$, $i=1,\ldots,k$. Assume that, for a given $i\in\{1,\ldots,k\}$, $p_{i}$ satisfies $Y_{0}h(p_{i})\neq 0$ and consider a small neighborhood $U_i$ of $p_i$. Accordingly, $p_i$ has one of the following types
\begin{enumerate}
	\item[(\textbf{$R_1$})] $\s\cap U_i\setminus\{p_{i}\}$ has a connected component contained in $\s^{c}$ and another in $\s^{s}$, and $\Gamma_0\cap W^{s,u}_{+}(p_{i})\neq \emptyset$ (see Figure \ref{typessing} (a));
	\item[(\textbf{$R_2$})] $\s\cap U_i\setminus\{p_{i}\}$ has a connected component contained in $\s^{c}$ and another in $\s^{s}$ and either $\Gamma_0\cap W^{s}_{+}(p_{i})= \emptyset$ or $\Gamma_0\cap W^{u}_{+}(p_{i})= \emptyset$ (see Figure \ref{typessing} (b,c));	
	\item[(\textbf{$R_3$})] $\s\cap U_i\setminus\{p_{i}\}\subset \s^{c}$ (see Figure \ref{typessing} (d)).
\end{enumerate}
The points $p_{i}$ satisfying $Xh(p_{i})\neq 0$ are classified analogously.

\begin{figure}[h!]
	\centering
	\bigskip
	\begin{overpic}[width=13cm]{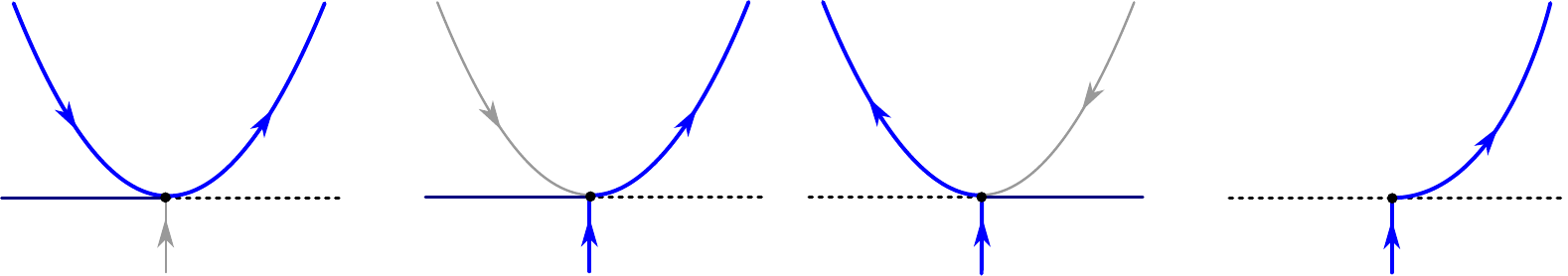}
		\put(8.5,-4){(a)}\put(35.5,-4){(b)}\put(61,-4){(c)}\put(86.5,-4){(d)}	
			\put(10,6){{\footnotesize $p_i$ }}
		\put(37,6){{\footnotesize $p_i$ }}	
		\put(62,6){{\footnotesize $p_i$ }}	
		\put(88,6){{\footnotesize $p_i$ }}	
		\put(20,2.5){{\footnotesize $\s$ }}		
	\end{overpic}
	\bigskip
	\caption{Types of local characterization of $\Gamma_0$ around the regular-tangential singularity $p_i$: Figure (a) and its time reversing illustrate type $R_1$; Figures (b,c) and their time reversing illustrate type $R_2$; Figure (d) and its time reversing illustrate type $R_3$. Bold lines represent the intersection $\Gamma_0\cap W^{s,u}_{+}(p_{i})$. Dashed lines represents $\s^c$. }\label{typessing}	
\end{figure} 

If $p_i$ is of type $R_1$, then we consider $\sigma_i(Z_0)=\{a_i\}\times(-\e_{i},+\e_{i})\cap M^{+}$. So, we can follow the case \textbf{(E-I)} from Section \ref{disp_sec} to construct the transfer functions $T_{i}^{u,s}: \sigma_{i}(X_0)\rightarrow \tau_{i}^{u,s}$ defined by the flow of $X_0$. Recall that $T_{i}^{s}$ and $T_{i}^{u}$ are restrictions of germs of diffeomorphisms (see Figure \ref{transition-I}).

\begin{figure}[h!]
	\centering
	\bigskip
	\begin{overpic}[width=4cm]{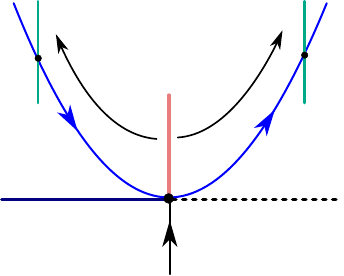}
								\put(52,16){{\footnotesize $p_i$ }}	
								\put(12,75){{\footnotesize $\tau_i^s$ }}	
								\put(80,75){{\footnotesize $\tau_i^u$ }}
								\put(30,55){{\footnotesize $T_i^s$ }}
								\put(60,55){{\footnotesize $T_i^u$ }}		
	\end{overpic}
	\bigskip
	\caption{  Construction of the maps $T_i^{u,s}$: type $R_1$. }\label{transition-I}	
\end{figure} 

If $p_i$ is of type $R_2$ or $R_3$, we consider the tangential section $\sigma_{i}(Z_0)=(a_{i}-\e_{i},a_{i}+\e_{i})\times\{0\} \cap \s^{c},$ where $\e_{i}$ is sufficiently small. So, we can follow the case \textbf{(O)} from Section \ref{disp_sec} to construct the transfer functions $T_i^{u}: \sigma_i(Z_0)\rightarrow \tau_i^{u}$ and $T_i^{s}: \sigma_i(Z_0)\rightarrow \tau_i^{s}$ induced by the flows of $X_0$ and $Y_0$, respectively. Notice that $T_i^s$ is the restriction of a germ of diffeomorphism and Theorem \ref{charac1} is applied to characterize $T_i^u$ (see Figure \ref{frII}).

\begin{figure}[h!]
	\centering
	\bigskip
	\begin{overpic}[width=11cm]{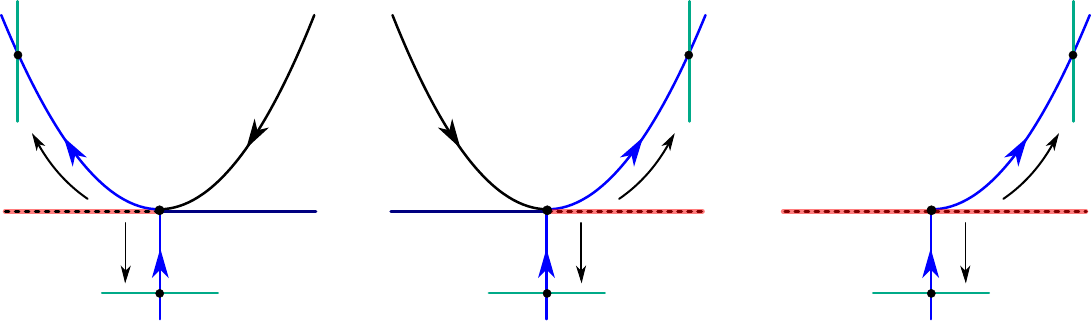}
		\put(14,12){{\footnotesize $p_i$}}
		\put(49,12){{\footnotesize $p_i$}}
		\put(84,12){{\footnotesize $p_i$}}
	\put(7,5){{\footnotesize $T_i^s$}}
\put(55,5){{\footnotesize $T_i^s$}}
\put(90,5){{\footnotesize $T_i^s$}}	
	\put(1,12){{\footnotesize $T_i^u$}}
\put(61,12){{\footnotesize $T_i^u$}}
\put(96,12){{\footnotesize $T_i^u$}}	
\put(2,28){{\footnotesize $\tau_i^u$}}	
\put(60,28){{\footnotesize $\tau_i^u$}}	
\put(95,28){{\footnotesize $\tau_i^u$}}	
\put(18,0){{\footnotesize $\tau_i^s$}}	
\put(54,0){{\footnotesize $\tau_i^s$}}	
\put(89.5,0){{\footnotesize $\tau_i^s$}}	
\put(13,-3){{\footnotesize $(a)$}}	
\put(48,-3){{\footnotesize $(b)$}}	
\put(85,-3){{\footnotesize $(c)$}}	
	\end{overpic}
	\bigskip
	\caption{ Construction of the maps $T_i^{u,s}$: types $R_2$ ($(a)$ and $(b)$) and $R_3$ $(c)$. }\label{frII}	
\end{figure}

Now, in order to describe the displacement functions associated with $\Gamma_0$, we characterize the unfolding of each tangential singularity. 

If $p_{i}$ is of type $R_1$, then $T_{i}^{s}$ and $T_{i}^{u}$ are germs of diffeomorphisms at $p_{i}$. So, as described in Section \ref{disp_sec}, for any $Z=(X,Y)\in\Or$ in a small neighborhood $\V_{i}$ of $Z_{0}$, there exist transfer functions $T_{i}^{s}(Z): \sigma_{i}(Z)\rightarrow \tau_{i}^{s}$ and $T_{i}^{u}(Z):\sigma_{i}(Z)\rightarrow \tau_{i}^{s}$ which are also germs of diffeomorphisms at $p_{i}$. From now on, we simplify the notation by omitting the dependence of functions and parameters on $Z$, except when it is necessary.

If $p_i$ is of type $R_2$ and $R_3,$ from Theorem \ref{charac2} there exists a neighborhood $\V$ of $Z_{0}$ such that for each $Z=(X,Y)\in\V$ the transfer function corresponding to $p_i$, for $i\in\{1,2,\ldots,k\},$ is given by
\begin{equation}\label{Tu}
	T^u_i(h^i(x))=\kappa_i (x-a_i)^{n_i}+ \sum_{j=0}^{n_i-2}\lambda_{j}^i(x-a_i)^{j}+\er((x-a_i)^{n_i+1}),
	\end{equation}
	where  $h^i=h^i_Z:(a_i-\e,a_i+\e)\rightarrow(a_i-\e,a_i+\e)\times\{0\}$ is a diffeomorphism, with $h^i_{Z_0}(a_i)=a_i,$ $\sgn(\kappa_{i}(Z))=\sgn(\kappa_{i}(Z_{0})),$ and $\lambda^i_j=\lambda^{i}_j(Z),$ for $j\in\{0,1,\ldots,n_{i}-2\},$   are parameters.

Notice that $T_i^s$ is a germ of diffeomorphism on $\sigma_i$. Thus, for $Z\in\V$ and for each $i=1,\cdots,k$ we have obtained two maps $T_{i}^{s,u}$ defined in a neighborhood of $p_{i}$ which describes the behavior of the orbits contained in $M^{+}$ connecting points of $\tau_{i}^{s,u}$ and $\sigma_{i}$. In addition, each transversal section $\tau_{i-1}^{u}$ is connected to $\tau_{i}^{s}$ via a diffeomorphism $D_{i}$ satisfying:
\begin{equation}
\label{invDzTs}[D_{i-1}]^{-1}\circ T_{i}^s(h^{i}(x))=\widetilde c_{i-1}+\widetilde d_{i-1}(x-a_{i})+\er_2(x-a_{i}),
\end{equation}
where $\widetilde c_{i-1}(Z_{0})=q_{i-1}^{u}$, and $\sgn(\widetilde d_{i-1}(Z))=\sgn(\widetilde d_{i-1}(Z_0))$ (see Figure \ref{dispfig}). Recall that, in the above expression, we are assuming that $Y_0h(p_i)\neq0$. The case $X_0h(p_i)\neq0$ follows analogously.

Now, let $\mathcal{A}$ be an open annulus around $\Gamma_0$ containing the sections $\sigma_{i}(Z_{0})$.
Using the above characterization of the transfer functions and their unfoldings and Definition \ref{disp_def} we obtain that:
\begin{equation*}
\Delta_{i}(h^i(x_{i}),h^{i+1}(x_{i+1}))= \overline{\Delta_{i}}(x_{i}^u,x_{i+1}^s) +\er_{N_i +1}(x_i^u) +\er_{M_i +1}(x_{i+1}^s),
\end{equation*} 

\noindent where $x_i^u=x_{i}-a_{i}$, $x_{i+1}^s=x_{i+1}-a_{i+1}$, and

\begin{equation*}
\overline{\Delta_{i}}(x_{i}^u,x_{i+1}^s)=\beta_{i}+ P_{i}^{N_{i}}(x_i^u)+Q_{i}^{M_{i}}(x_{i+1}^s).
\end{equation*}
Here, $\beta_i=\beta_{i}(Z)=\Delta_{i}(h^i(a_{i}),h^{i+1}(a_{i+1}))$ and satisfies $\beta_{i}(Z_0)=0$. In addition, $P_{i}^{N_{i}}$ and $Q_{i}^{M_{i}}$ are non-vanishing polynomials of degree $N_{i}\leq \max\{2,\ n_{i}-2\}$ and $M_{i}\leq \max\{2,\ n_{i+1}-2\}$ with coefficients depending on $Z$ and satisfying $P_{i}^{N_{i}}(0)=Q_{i}^{M_{i}}(0)=0.$

Finally, the crossing system \eqref{cross} is equivalent to the following system:
\begin{equation}\label{algsystem}
\left\{
\begin{array}{l}
\vspace{0.2cm}\overline{\Delta_{1}}(x_{1}^u,x_{2}^s)+\er_{N_1 +1}(x_1^u) +\er_{M_1 +1}(x_{2}^s)=0,\\
\vspace{0.2cm}\overline{\Delta_{2}}(x_{2}^u,x_{3}^s)+\er_{N_2 +1}(x_i^u) +\er_{M_2 +1}(x_{3}^s)=0,\\
\ \ \ \ \ \ \ \ \ \ \vdots\\
\vspace{0.2cm}\overline{\Delta_{k-1}}(x_{k-1}^u,x_{k}^s)+\er_{N_{k-1} +1}(x_{k-1}^u) +\er_{M_{k-1} +1}(x_{k}^s)=0,\\
\vspace{0.2cm}\overline{\Delta_{k}}(x_{k}^u,x_{1}^s)+\er_{N_k +1}(x_k^u) +\er_{M_1 +1}(x_{1}^s)=0,\\
\vspace{0.2cm}x_{i}^{s,u}=f^{s,u}_{i}(x_{i})=x_i-a_i,\ i=1,\cdots,k,\\
h^i(x_{i})\in\sigma_{i},\ i=1,\cdots,k.
\end{array}
\right.
\end{equation}

\subsection{$\s$-Polycycles having a unique regular-tangential singularity}\label{polyn_sec}

\

Without loss of generality, the following conditions characterize the nonsmooth vector fields $Z_0=(X_0,Y_0)$ which admit a $\s$-polycycle having a unique regular-tangential singularity of multiplicity $n$ (see Figure \ref{cycletangreg_fig}):
\begin{enumerate}[i)]
	\item There exists $p\in\s$ such that $X_{0}$ has a $n$-multiplicity contact with $\s$ at $p$, $n\geq 2$ and $Y_0h(p)\neq 0$.
	\item $W^{u}_{+}(p)$ intersects $\s^{c}$ at $q\neq p$ and the arc-orbit $\arcf{p}{q}{X_{0}}$  is contained in $M^{+}$; 
	\item $W^{s}_{-}(p)$ intersects $\s^{c}$ at $r\neq p$ and the arc-orbit $\arcf{p}{r}{Y_{0}}$ of $Y_0$ is contained in $M^{-}$;
	\item If $r\neq q$, there exists a regular orbit of $Z_0$ connecting $r$ and $q$. 		
\end{enumerate}
Accordingly, consider $\Gamma_0$ as the union of the arc-orbits $\arcf{p}{r}{Z_{0}}$, $\arcf{r}{q}{Z_{0}}$, and $\arcf{q}{p}{Z_{0}}$.

\begin{figure}[h!]
	\centering
	\bigskip
	\begin{overpic}[width=10cm]{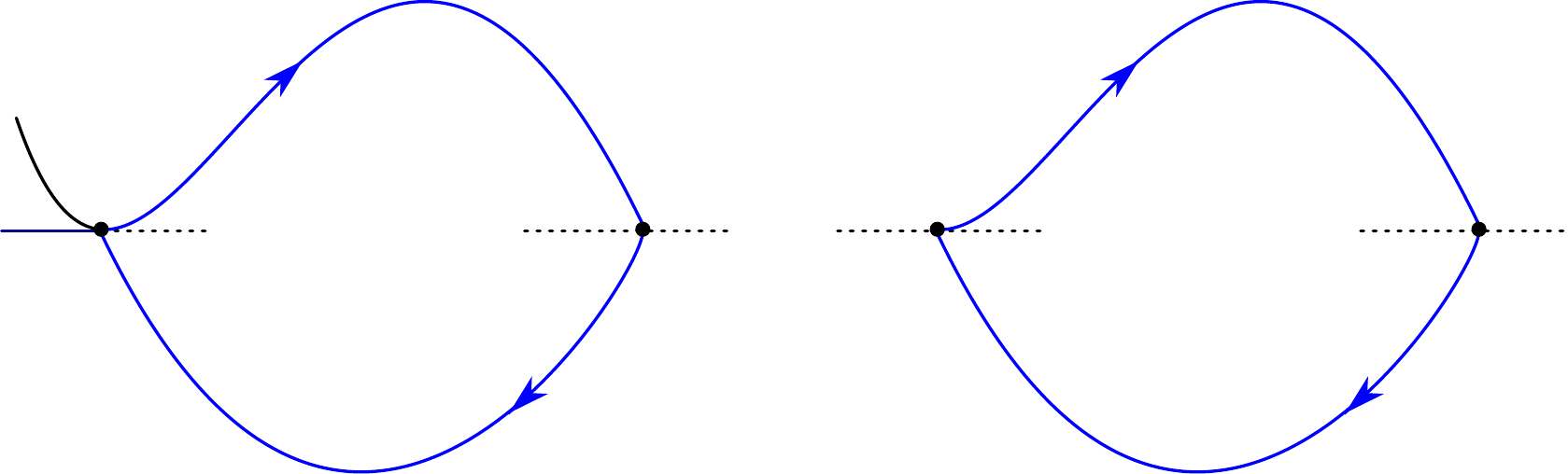}
		\put(23,-5){{\footnotesize $(a)$}}		
		\put(6,17){{\footnotesize $p$}}
		\put(59,17){{\footnotesize $p$}}
		\put(42,17){{\footnotesize $q$}}
\put(95,17){{\footnotesize $q$}}		
		\put(76,-5){{\footnotesize $(b)$}}	
	\end{overpic}
	\bigskip
	\caption{An example of $\s$-polycycle $\Gamma_0$ having a unique regular-tangential singularity of multiplicity $n$, when $q=r$, $(a)$ $n$ is even and when $(n)$ $n$ is odd.}	\label{cycletangreg_fig}
\end{figure} 	

Following the previous section for $k=1,$ $a_1=0,$ and $x_1=x_2=x$ the displacement function $\Delta:\sigma\rightarrow \R$  writes
\[
\begin{array}{rl}
\Delta(h(x))=& T^u(h(x))-[D]^{-1}\circ T^s(h(x))\vspace{0.1cm}\\
=&\displaystyle \lambda_{0}+\kappa x^{n}+ \sum_{j=1}^{n-2}\lambda_{j}x^{j}+\er(x^{n+1})\\
&-\widetilde{c}-\widetilde{d}x+\er_2(x),
\end{array}
\]
where $\sgn(\widetilde d(Z))=\sgn(\widetilde d(Z_0))$. Here, it is easy to see that assumptions $(i)$-$(iv)$ imply that $\widetilde d(Z_0)<0$. Taking
\begin{equation}\label{eta}
\beta=\lambda_0-\widetilde{c},\quad \lambda=\lambda_{1}-\widetilde{d},\quad \text{and}\quad \eta=(\beta,\lambda),
\end{equation}
the displacement function $\Delta(h(x))$ writes
\begin{equation}\label{disp1_eq}
\Delta(h(x))= \beta+\lambda x +\er_{2}(x).
\end{equation}
Notice that $\eta:\V\rightarrow V$ is a surjective function onto a small neighborhood $V$ of $(0,-\widetilde{d}(Z_0))$ satisfying $\beta(Z_0)=0$ and $\lambda(Z_0)=-\widetilde d(Z_0)\neq0$. In this case, the crossing system \eqref{algsystem} is reduced to the equation $\beta+\lambda x +\er_{2}(x)=0$, $h(x)\in\sigma(Z)$.

As a first result on the $\s$-polycycle $\Gamma_0$ we have the following proposition.

\begin{prop}\label{stab_prop}
Let $\Gamma_0$ be a $\s$-polycycle having a unique regular-tangential singularity of multiplicity $n$ satisfying $(i)$-$(iv)$. Then, $\Gamma_0$ attracts the orbits passing through the section $\sigma(Z_0)$ (domain of $T^u(Z_0)$). In this case, we say that $\Gamma_0$ is C-attractive.
\end{prop}
 \begin{proof}
	Notice that the first return map associated with the $\s$-polycycle $\Gamma_0$ of $Z_0$ is given by $\mathcal{P}_0(x)=\big([D]^{-1}\circ T^s\big)^{-1}\circ T^u(x)$, where, from \eqref{Tu} and \eqref{invDzTs} (recall that $h_{Z_0}=Id$),
	\[
	T^u(x)=\kappa x^n+\er_{n+1}(x)  \quad\text{and}\quad  [D]^{-1}\circ T^s(x)=\widetilde d\, x+\er_{2}(x).
	\]
	Hence, 
	\[
	\mathcal{P}_0(x)=\dfrac{\kappa}{\widetilde d}\,x^n+\er_{n+1}(x).
	\]
	Therefore, for $x$ small enough, $|\mathcal{P}_{0}(x)|<|x|$, which means that $\Gamma_0$ attracts the orbits passing through the section $\sigma(Z_0)$ (domain of $T^u(Z_0)$).  
\end{proof}

In what follows we state the main result of this section.

\begin{prop} \label{ngeral_prop}
	Let $Z_{0}$ be a nonsmooth vector field having a $\s$-polycycle $\Gamma_0$ containing a unique regular-tangential singularity of multiplicity $n$ satisfying $(i)$-$(iv)$. Then, the following statements hold.
	\begin{enumerate}[i)]
		\item There exist an annulus $\mathcal{A}_0$ at $\Gamma_0$ and a neighborhood $\V$ of $Z_0$ such that each $Z\in\V$ has at most one crossing limit cycle bifurcating from $\Gamma_0$ in $\mathcal{A}_0$, which is hyperbolic and attracting.
		\item Let $Z_{\beta,\lambda}$ be a continuous 2-parameter family in $\V$ such that $Z_{0,-\widetilde{d}(Z_0)}=Z_0$ and satisfying $Z_{\beta,\lambda}\in\eta^{-1}(\beta,\lambda),$ for every $(\beta,\lambda)\in V$. Then, for each $\lambda$ near $-\widetilde{d}(Z_0)$ and for each connected component  $C$ of $\sigma(Z_{\beta,\lambda})\cap \mathcal{A}_0$, there exists a non-empty open interval $I_{\lambda,C}$, satisfying $I_{\lambda,C}\times\{\lambda\}\subset V$, such that $Z_{\bg,\lambda}$ has a hyperbolic attracting crossing limit cycle passing through $C$, for each $\bg\in I_{\lambda,C}$.
	\end{enumerate}	  
\end{prop}		
 \begin{proof}
	Consider the function $\eta:\V\rightarrow V$ given by \eqref{eta}. For each $Z=(X,Y)\in\V$, we associate the displacement function $\Delta$ given in \eqref{disp1_eq}. From Section \ref{disp_sec}, we have that there exists $\e>0$ such that, for each $Z\in\V$, there exists a function $\widetilde{\Delta}:(-\e,\e)\rightarrow \R$ which is an extension of $\Delta$.
	
	Define the $\Cr$ function $\mathcal{F}: \V\times V\times (-\e,\e)\rightarrow \R$ as
	$$\mathcal{F}(Z,\beta,\lambda,x)=\widetilde{\Delta}(h(x))-\bg(Z)-\lambda(Z)x+\bg+\lambda x,$$
	and notice that
	$$\mathcal{F}(Z_0,0,-\widetilde{d}(Z_0), 0)=0,\quad \text{and}\quad \partial_x\mathcal{F}(Z_0,0,-\widetilde{d}(Z_0), 0)=-\widetilde{d}(Z_0)\neq 0.$$
	
	From the Implicit Function Theorem for Banach Spaces and reducing $\V$ and $V$ if necessary, there exists a unique $\Cr$ function $\mathcal{X}:\mathcal{V}\times V\rightarrow (-\e,\e)$ such that $\mathcal{F}(Z,\beta,\lambda,x)=0$ if, and only if, $x=\mathcal{X}(Z,\beta,\lambda)$.
	
	Since $$\mathcal{F}(Z,\bg,\lambda, x)=\bg+\lambda x+\er_2(x),$$
	it follows that $\mathcal{X}(Z,0,\lambda)=0$, for every $(Z,0,\lambda)\in\V\times V$. Consequently, we can see that
	\begin{equation}\label{zerogeral}\mathcal{X}(Z,\bg,\lambda)=-\dfrac{\bg}{\lambda}+\er_2(\bg).	\end{equation}
	
	It follows from the definition of the function $\mathcal{F}$ that
	\begin{equation}\label{zero}
	\mathcal{X}^*(Z)=\mathcal{X}(Z,\bg(Z),\lambda(Z))
	\end{equation}
	is the unique zero of $\widetilde{\Delta}(Z)$ in $(-\e,\e)$. Hence, $\Delta(Z)$ has at most one zero in $\sigma(Z)$. Moreover, since
	\[
	\dfrac{\partial \widetilde\Delta(Z_0)}{\partial x}(\mathcal{X}^*(Z_0))=-\widetilde d(Z_0)>0,
	\]
	it follows from \eqref{disp1_eq} that
	\[
	\dfrac{\partial \widetilde\Delta(Z)}{\partial x}(\mathcal{X}^*(Z))=\la(Z)+\er_2(\mathcal{X}^*(Z))>0,
	\]
	 for $Z$ sufficiently near $Z_0$. Therefore, the crossing limit cycle is hyperbolic and  attracting (from construction). The proof of item $(i)$ follows by taking $\mathcal{A}_0=\{p\in M; d(p,\Gamma_0)<\e\}$, where $d$ denotes the Hausdorff distance.
	
	Now, consider the family $Z_{\bg,\lambda}$ given in item $(ii)$. The unique zero of $\widetilde{\Delta}(Z_{\bg,\lambda})$ is given by
	\begin{equation}\label{xstar}
	x^*(\beta,\lambda)=\mathcal{X}^*(Z_{\beta,\lambda})=-\dfrac{\beta}{\lambda}+\er_2(\bg).
	\end{equation}
Recall that each isolated zero, $x_0$, of $\Delta(Z_{\beta,\lambda})$ is either a crossing limit cycle (if $x_0\in\textrm{int}(\sigma(Z_{\bg,\lambda}))$) or a $\s$-polycycle (if $x_0\in\partial\sigma(Z_{\bg,\lambda})$). So, let $C=(a,b)$ be a connected component of $\sigma(Z_{\bg,\lambda})\subset(-\e,\e)$ for some fixed parameter $\lambda\in\pi_{2}(V).$ Hence, from \eqref{xstar}, there exists a non-empty open interval $I_{\lambda,C}$  such that $I_{\lambda,C}\times\{\lambda\}\subset V$ and $x^*(\bg,\lambda)\in \textrm{int}(C)$ whenever $\bg\in I_{\lambda,C}$.			
\end{proof}

 \begin{remark} 
If we change the roles of $s$ and $u$ in the assumptions $(i)$-$(iv)$ in order to reverse the orientation of the cycle, all the results remain the same reversing the stability.
\end{remark}

Let $Z$ be a nonsmooth vector field sufficiently near $Z_0,$ and consider $\mathcal{X}^*(Z)$ given by \eqref{zero}. Propositions \ref{stab_prop} and \ref{ngeral_prop} provides the following possibilities for the  crossing dynamics in a small annulus $\mathcal{A}_0$ of $\Gamma_0$:
\begin{enumerate}[i)]
	\item if $\mathcal{X}^*(Z)\notin \sigma(Z)$, then $Z$ has no crossing limit cycles or $\s$-polycycles;
	\item if $\mathcal{X}^*(Z)\in \textrm{int}(\sigma(Z))$, then $Z$ has a unique crossing limit cycle with the same stability of $\Gamma_0$;
	\item if $\mathcal{X}^*(Z)\in \partial\sigma(Z)$, then $Z$ has a unique $\s$-polycyle containing $m\leq n-1$ regular-tangential singularities of multiplicity $n_{i}$, with $\sum_{i=1}^m n_{i}\leq n$.			
\end{enumerate}
In addition, items (i) and (ii) occur in open regions of the parameter space and item (iii) occurs in a hypersurface of the parameter space.

\subsection{$\s$-Polycycles having a unique regular-cusp singularity }\label{cusp_sec}

\

In the previous section, assuming that $Z_0$ has a $\s$-polycycle $\Gamma_0$ admitting a unique regular-tangential singularity of multiplicity $n$, $n\geq 2,$ we have identified all the possible crossing behavior of nonsmooth vector fields $Z=(X,Y)$ sufficiently near $Z_0$ in a small annulus $\mathcal{A}_0$ of $\Gamma_0.$
Nevertheless, the domain $\sigma(Z),$ of the displacement function \eqref{disp1_eq}, has some particularities depending on the multiplicity $n$. 
In order to illustrate it, we describe the bifurcation diagram of $Z_0$ around $\Gamma_0$ assuming that $n=3$. 

As before, the displacement function writes
\begin{equation}\label{cusp_disp}
\Delta(x)= \beta +\lambda x +\er_{2}(x).
\end{equation}
For the sake of simplicity, we are omitting the parametrization $h(x)$ in the displacement function \eqref{cusp_disp}. 

As we have shown before, $\Delta(Z)$ has a unique zero $x^*(\eta(Z))$ in an interval $(-\e,\e)$. Now, we have to study how the domain $\sigma(Z)$ of $\Delta(Z)$ changes with $Z$.
Now, we use the parameter $\la(Z)$, defined in \eqref{eta}, to characterize $\sigma(Z)$. Recall that $\lambda=\widetilde d-\lambda_1$ and $\lambda_1$ is given in the unfolding of $T^u_{Z_0}$. Analogously to the proof of Theorem \ref{charac2}, we consider a coordinate system $(\bar x,\bar y)$ which trivializes the flow of $X$ at $(0,0).$ In this coordinate system, $\s=\{(\bar x, \gamma_{\lambda_1}(\bar x)); \bar x\in(-\e,\e)\}$ and the transition map $T^u$ becomes $T^u_*(\bar x)=\gamma_{\lambda_1}(\bar x)$, where $\gamma_{\lambda_1}(\bar x)= \kappa \bar{x}^{3}+\lambda_1 \bar x+\er_{4}(\bar x)$ and $\kappa(Z_0)=-X_{0}^{3}h(0,0).$

There is no loss of generality in assuming that $\kappa(Z_0)<0$, since the case $\kappa(Z_0)>0$ is completely analogous. Hence, we have the following situation (see Figure \ref{cusp_unf_fig}):
\begin{figure}[h!]
	\centering
	\bigskip
	\begin{overpic}[width=13cm]{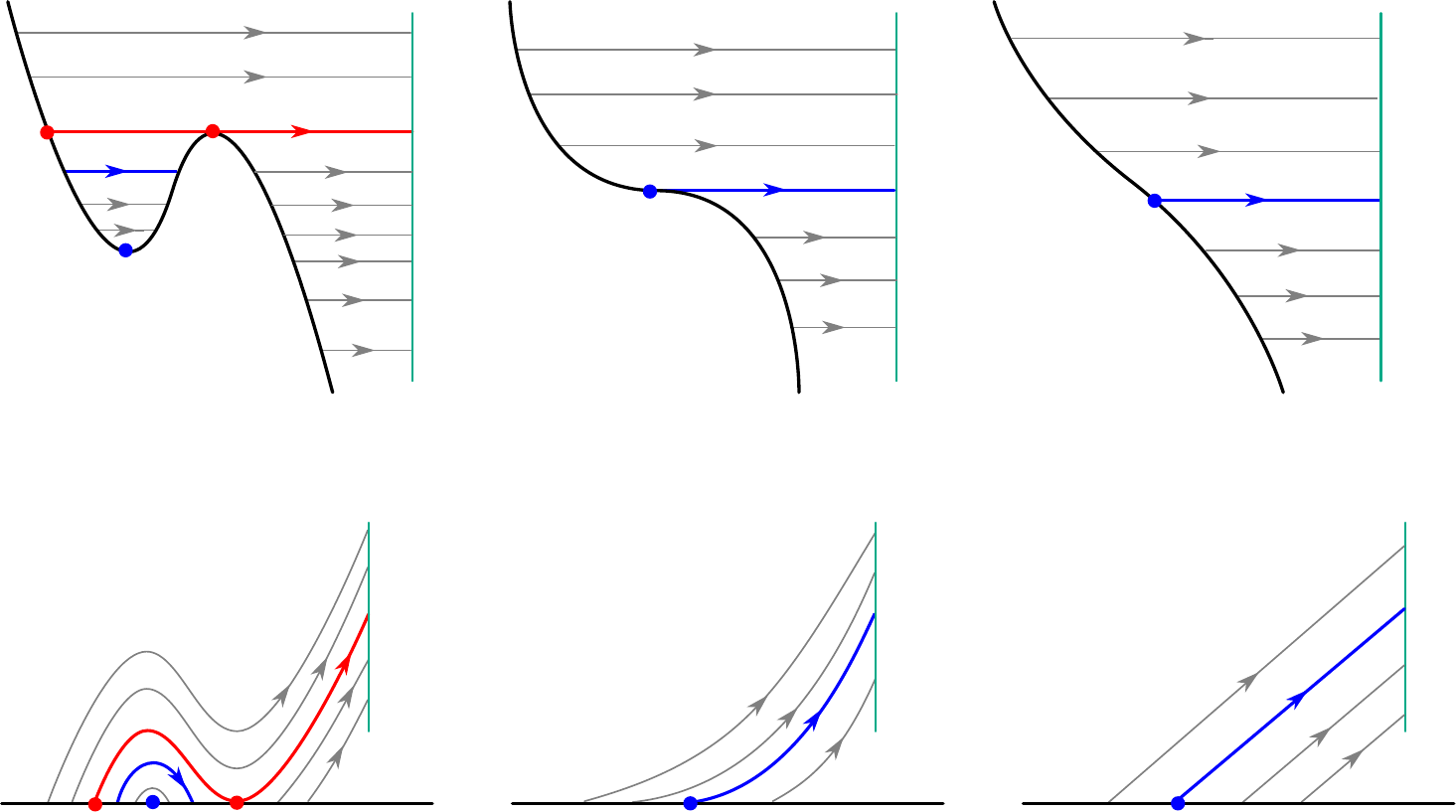}
		\put(11,-5){{\footnotesize $\lambda_1(Z)>0$}}		
		\put(45,-5){{\footnotesize $\lambda_1(Z)=0$}}		
		\put(78,-5){{\footnotesize $\lambda_1(Z)<0$}}	
		\put(28,-2){{\footnotesize $\s$}}
		\put(19,30){{\footnotesize $\s$}}		
		\put(26,18){{\footnotesize $\tau^u$}}
		\put(61,18){{\footnotesize $\tau^u$}}
		\put(97,18){{\footnotesize $\tau^u$}}
		\put(29,53){{\footnotesize $\tau^u$}}
\put(62,53){{\footnotesize $\tau^u$}}
\put(96,53){{\footnotesize $\tau^u$}}				
		\put(3,-2){{\scriptsize $A(\lambda_1)$}}
		\put(9,-2){{\scriptsize $I(\lambda_1)$}}
		\put(15,-2){{\scriptsize $V(\lambda_1)$}}
		\put(47,-2){{\scriptsize $\vec{0}$}}
				\put(80,-2){{\scriptsize $\vec{0}$}}
		\put(-5,10){{\footnotesize $(b)$}}
\put(-5,40){{\footnotesize $(a)$}}			
	\end{overpic}
	\bigskip
	\vspace{0.2cm}
	\caption{Unfolding of the regular-cusp singularity in the coordinate system $(a)$ $(\overline{x},\overline{y})$ and $(b)$ $(x,y)$.}	\label{cusp_unf_fig}
\end{figure} 				
\begin{enumerate}[i)]
	\item If $\lambda_1<0$, all the orbits of $X$ are transversal to $\s$, Therefore, $\sigma=(-\e,\e)$;
	\item If $\lambda_1=0$, $\sigma=(-\e,\e)$ (see Corollary \ref{charac_sigma});
	\item If $\lambda_1>0$, then $\cg_{\lambda_1(Z)}(\bar x)$ has a minimum at $I(\lambda_1)=-\sqrt{-\dfrac{\lambda_1}{3\kappa}}+ \er_1(\lambda_1)$  and a maximum at $V(\lambda_1)=\sqrt{-\dfrac{\lambda_1}{3\kappa}}+ \er_1(\lambda_1)$. Therefore, $X$ has a visible regular-fold singularity at $V(\lambda_1)$  and an invisible regular-fold singularity at $I(\lambda_1)$. In addition, the orbit passing through the visible regular-fold singularity intersects $\s$ backward in time at a point $A(\lambda_1)<I(\lambda_1)$. This means that $\sigma(Z)=(-\e,A(\lambda_1)]\cup[V(\lambda_1),\e)$ and $A(\lambda_1),V(\lambda_1)\rightarrow 0$ as $\lambda_1\rightarrow 0^+$.			
\end{enumerate}

 From the discussion above we have the following result.
\begin{theorem}\label{regcuspcycle}
 	Let $Z_{0}$ be a nonsmooth vector field having a C-attracting $\s$-polycycle $\Gamma_0$ containing a unique regular-cusp singularity.
 	Therefore, there exists an annulus $\mathcal{A}_0$ around $\Gamma_0$ such that for each annulus $\mathcal{A}$, with $\Gamma_0\subset\mathcal{A}\subset \mathcal{A}_0,$ there exist neighborhoods $\V\subset\Or$ of  $Z_0$ and $V\subset\R^2$ of $(0,0)$, a surjective function  $(\beta,\lambda_1):\V\rightarrow V,$ with $(\beta,\lambda_1)(Z_0)=(0,0),$ and three smooth functions $\overline{A},\overline{V},\overline{I}:\V\rightarrow R$ with $\overline{A}(Z_0)=\overline{V}(Z_0)=\overline{I}(Z_0)=0,$ for which the following statements hold inside $\mathcal{A}$.
 	\begin{enumerate}
 		\item If $\lambda_1<0$, then $Z$ has a unique crossing limit cycle of $Z$, which is hyperbolic attracting. 
 		\item If $\lambda_1=0$ and $\beta(Z)\neq 0$,  then $Z$ has  a unique crossing limit cycle of $Z$, which is hyperbolic attracting. 
 		\item If $\lambda_1=\beta=0$, then  $Z$ has a unique $\s$-polycycle, containing a unique regular-cusp singularity of $Z$, which is C-attracting. 			
 		\item If $\lambda_1>0$ and $\beta>\overline{V}$, then  $Z$ has a unique crossing limit cycle of $Z$, which is hyperbolic attracting. 
 		\item If $\lambda_1>0$ and $\beta=\overline{V}$, then  $Z$ has a unique $\s$-polycycle, containing a visible unique regular-fold singularity, which is C-attracting. 
 		\item If $\lambda_1>0$ and $\overline{V}<\beta<\overline{I}$, then  $Z$ has a sliding cycle containing a visible regular-fold singularity.
 		\item If $\lambda_1>0$ and $\overline{I}=\beta(Z)$, then  $Z$ has a  sliding cycle containing a visible regular-fold singularity and an invisible regular-fold  singularity.
 		\item If $\lambda_1>0$ and $\overline{A}<\beta<\overline{I}$, then  $Z$ has a sliding cycle containing a unique visible regular-fold singularity.				
 		\item If $\lambda_1>0$ and $\beta=\overline{A}$, then  $Z$ has a unique $\s$-polycycle, containing a unique regular-fold singularity, which is C-attracting. 
 		\item If $\lambda_1>0$ and $\overline{A}<\beta$, then  $Z$ has a unique crossing limit cycle of $Z$, which is hyperbolic attracting. 
 	\end{enumerate}	
 In addition, 
 \begin{equation}\label{AVI}
 \begin{array}{rl}
 \overline{A}=&\widetilde{d} A(\lambda_1)+\er_{2}(\lambda_1,A(\lambda_1)),\vspace{0.2cm}\\
  \overline{V}=&\widetilde{d}\sqrt{-\dfrac{\lambda_1}{3\kappa}}+ \er_1(\lambda_1),\vspace{0.2cm}\\
   \overline{I}=&-\widetilde{d}\sqrt{-\dfrac{\lambda_1}{3\kappa}}+\er_1(\lambda_1),
   \end{array}
   \end{equation}
   where $A(\lambda_1)$ and $V(\lambda_1)$ are defined as the extrema of  $\sigma(Z)$ as follows $\sigma(Z)=(-\e,A(\lambda_1)]\cup[V(\lambda_1),\e)$. 
 \end{theorem}
  
 The theorem above provides the bifurcation diagram of $Z_0$ in the $(\beta,\lambda_1)$-parameter space (see Figure \ref{bifcr1}). 
 
 \begin{proof}
 From the construction of the crossing system \eqref{algsystem}, performed in Section \ref{descripregtan}, we get the existence of an annulus $\mathcal{A}_0$ around $\Gamma_0$ and neighborhoods $\V_0\subset\Or$ of  $Z_0$ and $V_0\subset\R^2$ of $(0,0)$, for which the equation \eqref{cusp_disp} is well defined.
 	
	Now, given an annulus $\mathcal{A}$, with $\Gamma_0\subset\mathcal{A}\subset \mathcal{A}_0,$ let $\e>0$ satisfy $(-\e,\e)\times \{0\}\subset \mathcal{A}$. Consider the function $\mathcal{X}:\V_0\times V_0\rightarrow (-\e,\e)$ given by \eqref{zerogeral}, and for a sufficiently small neighborhood $U\subset \rn
	2$ of the origin, define $\mathcal{B}:\V\times U\times (-\e,\e)\rightarrow \R$ by
	$$\mathcal{B}(Z,\bg,\lambda_1,v)=\mathcal{X}(Z,\bg,\lambda_1-\widetilde{d}(Z))-v= -\dfrac{\bg}{\lambda_1-\widetilde{d}(Z)}-v+\er_2(\bg).$$
	Notice that 
$$\mathcal{B}(Z_0,0,0, 0)=0,\quad \text{and}\quad \partial_\bg\mathcal{B}(Z_0,0,0, 0)=\dfrac{1}{\widetilde{d}(Z_0)}\neq 0.$$	
	From the Implicit Function Theorem for Banach Spaces, there exist $\dg>0$, an open interval $J$ containing $0,$ and a unique $\Cr$ function $\beta^*:\V\times J\times (-\e,\e)\rightarrow (-\dg,\dg)$ such that $\mathcal{B}(Z,\bg,\lambda,v)=0$ if, and only if $\bg=\beta^*(Z,\lambda_1,v)$.
	Also, we can see that
	$$\beta^*(Z,\lambda_1,v)=\widetilde{d}(Z)v-\lambda_1 v+\er_2(v).$$

	Notice that, if $\overline{A}(Z)=\beta^*(Z,\lambda_1(Z),A(\lambda_1(Z)))$ and $\overline{V}(Z)=\beta^*(Z,\lambda_1(Z),V(\lambda_1(Z))),$ then $\mathcal{X}^*(Z,\overline{A}(Z),\lambda_1(Z))=A(\lambda_1(Z))$ and $\mathcal{X}^*(Z,\overline{V}(Z),\lambda_1(Z))=V(\lambda_1(Z))$.	 Since 
	\[
	V(\lambda_1(Z))=\sqrt{-\dfrac{\lambda_1(Z)}{3\kappa(Z)}}+ \er_1(\lambda_1(Z)),
	\] 
	we get $\overline{V}(Z)$ from \eqref{AVI}.	
	
	From construction of the maps $T^u(Z)$, $T^s(Z)$ and $D(Z)$ given in \eqref{Tu} and \eqref{invDzTs}, it follows that the points $I(\lambda_1(Z))$ and $V(\lambda_1(Z))$ are connected by an orbit of $Z=(X,Y)$ if, and only if,
$$G(Z)=:T^u(Z)(V(\lambda_1(Z)))-[D(Z)]^{-1}\circ T^s(Z)(I(\lambda_1(Z)))=0.$$
Notice that
$$G(Z)=\bg(Z)+\widetilde{d}(Z)\sqrt{-\dfrac{\lambda_1(Z)}{3\kappa(Z)}}+\er_1(\lambda_1(Z)).$$
Thus, applying the Implicit Function Theorem to the function $\mathcal{G}:\V\times  (-\dg,\dg)\rightarrow \rn{},$ given by $\mathcal{G}(Z,\bg):=G(Z)-\bg(Z)+\bg,$ at the point $(Z_0,0)$, we get a unique $\Cr$ function $\overline{I}:\V\rightarrow (-\dg,\dg)$ such that $\mathcal{G}(Z,\bg)=0$ if, and only if, $\bg=\overline{I}(Z)$. Hence, the points $V(\lambda_1(Z))$ and $A(\lambda_1(Z))$ are connected by an orbit of $Z$ if, and only if $\bg(Z)=\overline{I}(Z)$.
In this case,
$$\overline{I}(Z)=-\widetilde{d}(Z)\sqrt{-\dfrac{\lambda_1(Z)}{3\kappa(Z)}}+\er_1(\lambda_1(Z)).$$

From here, the proof follows directly from the definitions of the curves $\overline{A}, \overline{V}$ and $\overline{I}$, and Propositions \ref{stab_prop} and \ref{ngeral_prop}.
\end{proof}

In order to illustrate these results one provides a practical model realizing such bifurcation diagram.  \begin{example}\label{example1}
		Consider the Filippov vector field 
		\begin{equation}\label{model-cusp}
		Z_{\lambda_1,\beta}(x,y)=\left\{  \begin{array}{ll}
			X_{\lambda_1}(x,y)=(1,-4x^3+6x^2+2\lambda_1(x-1))^T,&y>0,\\
			Y_{\beta}(x,y)=(-1,2(1-x)-\beta)^T,&y<0.
		\end{array}  \right.
	\end{equation}
Notice that this is a piecewise Hamiltonian vector field, with Hamiltonian maps give, respectively,  by 
\[H^+_{\lambda_1}(x,y)= (x^3+\lambda_1x)(x-2)+y  \quad \mbox{ and } \quad H^-_{\beta}(x,y)= x(x-2+\beta)-y.\]
The vector field \eqref{model-cusp} satisfies:
\begin{enumerate}[(i)]
	\item $X_0$ has a cusp point at the origin and an invisible fold point at \((3/2,0)\);
	\item the trajectory of $X_0$ through the origin crosses \(\s=\{y=0\}\) again at \((2,0)\) and this arc is contained in \(\s^+=\{y>0\}\);
	\item the parameter \(\lambda_1\) unfolds the cusp point of \(X_0\) by creating two fold points for \(\lambda_1>0\);
	\item $Y_0$ has an invisible fold point at \((1,0)\) and its trajectory through \((2,0)\)  crosses \(\s\) again at the origin and this arc contained in \(\s^-=\{y<0\}\). It means that $Z_{0,0}$ has a \(\s\)-polycycle through a unique cusp-regular point; 
	\item the parameter \(\beta\) changes the position of the fold point of \(Y_{\beta}\) and this motion breaks the connection of the \(\s\)-polycycle;
	\item by changing the parameters \(\lambda_1\) and \(\beta\) it is possible to obtain all configurations showed at the proposed bifurcation diagram at Figure \ref{bifcr1}.
\end{enumerate}

\end{example}

\subsection{$\s$-Polycycles having several regular-tangential singularities}\label{polyn+_sec}

\

	Now we perform an analysis of a class of $\s$-polycycles having several regular-tangential singularities and we obtain similar results for those in Section \ref{polyn_sec}.  Consider the class of nonsmooth vector fields $Z_0=(X_0,Y_0)$ which admit a $\s$-polycycle having $k$ regular-tangential singularities, $p_i\in\s$, of multiplicity $n_i$, $i=1,\ldots,k$ satisfying the following property:
	\begin{enumerate}[(A)]
		\item for each $i=1,\ldots,k$, there exists a curve $\gamma_i$ connecting $p_i$ and $p_{i+1}$, oriented from $p_i$ to $p_{i+1}$, 
		such that $\gamma_i\setminus\{p_i,p_{i+1}\}$ is a regular orbit of $Z_0$, $\gamma_i$ is tangent to $\s$ at $p_i$ and transversal to  $\s$  at $p_{i+1}$, where $p_{k+1}=p_1$ (see Figure \ref{cycletangreg+_fig}).
	\end{enumerate}

In what follows, without loss of generality, we assume that  $h(x,y)=y$, $p_1=(0,0)$, and  $p_i=(a_i,0)$, $i=2,\cdots,k$.

\begin{figure}[h!]
	\centering
	\bigskip
	\begin{overpic}[width=12cm]{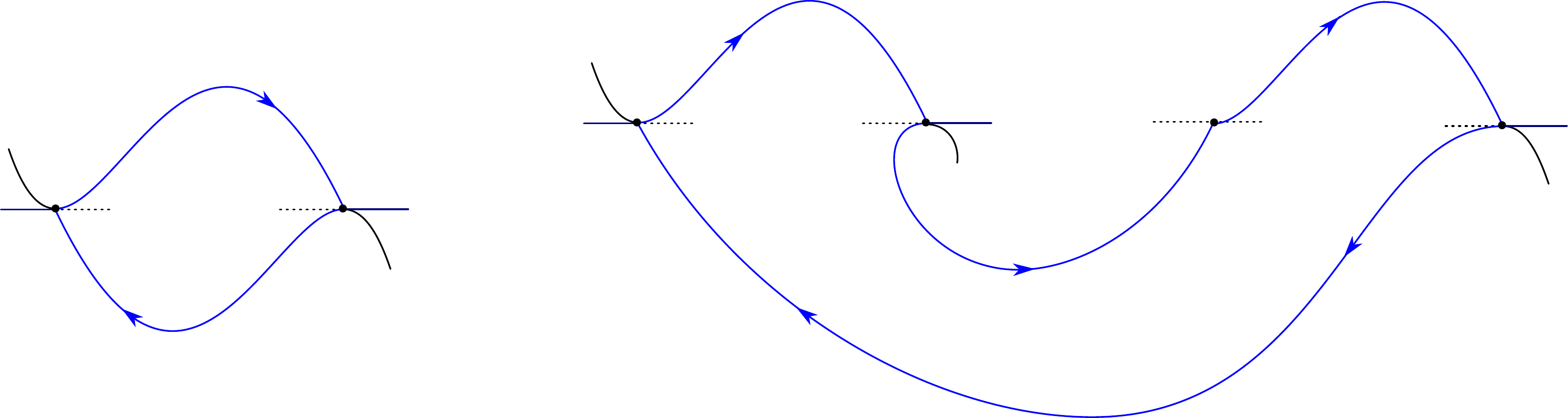}
		\put(1.5,11){{\footnotesize $p_1$}}
		\put(-5,12){{\footnotesize $\s$}}
		\put(20,11){{\footnotesize $p_2$}}
		\put(58,17){{\footnotesize $p_2$}}
		\put(42,17){{\footnotesize $p_1$}}
		\put(78,17){{\footnotesize $p_3$}}	
		\put(94,17){{\footnotesize $p_4$}}		
	\end{overpic}
	\bigskip
	\caption{$\s$-polycycles satisfying hypothesis $(A)$.}	\label{cycletangreg+_fig}
\end{figure} 
	
	Following the constructions presented in Sections \ref{descripregtan} and \ref{polyn_sec} the displacement functions $\Delta_i:\sigma_i\rightarrow \R$ are given by
	
	\begin{equation*}
		\Delta_{i}(h^i(x_{i}),h^{i+1}(x_{i+1}))= \beta_{i}+\widetilde{\lambda}_i x^s_{i+1} +\er_{2}(x_i^u) +\er_{2}(x_{i+1}^s),
	\end{equation*} 
	where $x_i^{s,u}=x_{i}-a_{i}$, $\beta_{i}=\Delta_{i}(h^i(a_{i}),h^{i+1}(a_{i+1}))$ satisfies $\beta_{i}(Z_0)=0,$ and $\widetilde{\lambda}_i=\lambda^i_{1}-\widetilde{d}_i$ satisfies $\widetilde{\lambda}_i(Z_0)=-\widetilde{d}_i(Z_0)\neq0$.
	Thus, there exists a neighborhood $\V$ of $Z_0$ such that for each $Z\in\V$ and $i=1,\ldots,k$,  $\widetilde{\lambda}_i=-\widetilde{d}_i\neq0$ and the crossing system \eqref{algsystem} is given by
	
	\begin{equation}\label{dispsystem+}
		\left\{
		\begin{array}{l}
			\vspace{0.2cm}\Delta_{1}(h^1(x_{1}),h^{2}(x_{2}))= \beta_{1}+\widetilde{\lambda}_1 x^s_{2} +\er_{2}(x_1^u) +\er_{2}(x_{2}^s)=0,\\
			\vspace{0.2cm}\Delta_{2}(h^2(x_{2}),h^{3}(x_{3}))= \beta_{2}+\widetilde{\lambda}_2 x^s_{3} +\er_{2}(x_2^u) +\er_{2}(x_{3}^s)=0,\\
			\vdots\\
			\vspace{0.2cm}\Delta_{k-1}(h^{k-1}(x_{k-1}),h^{k}(x_{k}))= \beta_{k-1}+\widetilde{\lambda}_{k-1} x^s_{k} +\er_{2}(x_{k-1}^u) +\er_{2}(x_{k}^s)=0,\\
			\vspace{0.2cm}\Delta_{k}(h^k(x_{k}),h^{1}(x_{1}))= \beta_{k}+\widetilde{\lambda}_k x^s_{1} +\er_{2}(x_k^u) +\er_{2}(x_{1}^s)=0,\\
			\vspace{0.2cm}x_{i}^{s,u}=x_i-a_i,\ i=1,\cdots,k,\\
			h^i(x_{i})\in\sigma_{i},\ i=1,\cdots,k.
		\end{array}
		\right.
	\end{equation}

	So for the $\s$-polycycle $\Gamma_0$ we have the following proposition.
	
	\begin{prop}\label{stab+_prop}
		Let $\Gamma_0$ be a $\s$-polycycle having $k$ regular-tangential singularities $p_i\in\s$, of multiplicity $n_i$, $i=1,\ldots,k,$ satisfying the property $(A)$. Then, $\Gamma_0$ attracts the orbits passing through the section $\sigma_1(Z_0)$ (domain of $T_1^u(Z_0)$).  In this case, we say that $\Gamma_0$ is C-attracting.
	\end{prop}
 \begin{proof}
			Notice that the first return map associated with the $\s$-polycycle $\Gamma_0$ of $Z_0$ is given by $$\mathcal{P}_0(x)=\big([D_k]^{-1}\circ T_1^s\big)^{-1}\circ T_k^u\circ\big([D_{k-1}]^{-1}\circ T_k^s\big)^{-1}\circ T_{-1}^u\circ\cdots\circ\big([D_1]^{-1}\circ T_2^s\big)^{-1}\circ T_1^u(x)$$ 
			where, from \eqref{Tu} and \eqref{invDzTs} (recall that $h^ 1_{Z_0}=Id$),
			\[
			T^u_i(x)=\kappa_ix^{n_i}+\er_{n_i+1}(x)  \quad\text{and}\quad [D_{i-1}]^{-1}\circ T_i^s(x)=\widetilde d_i\, x+\er_{2}(x).
			\]
			Hence, 
			\[
			\mathcal{P}_0(x)=\prod_{i=1}^{k}\dfrac{\kappa_i}{\widetilde{d}_i}\,x^{n_i}+\er_{N}(x),\quad  N=n_1+n_2+\cdots+ n_k+1.
			\]
			Therefore, for $|x|$ small enough, $|\mathcal{P}_{0}(x)|<|x|$, which means that $\Gamma_0$ attracts the orbits passing through the section $\sigma_1(Z_0)$ (domain of $T_1^u(Z_0)$).  
	\end{proof}
	
	Set $\eta=(\beta,\widetilde{\lambda})$ with $\beta=(\beta_1,\ldots,\beta_k)$ and $\widetilde\lambda=(\widetilde{\lambda}_1,$ $\ldots,\widetilde{\lambda}_k)$, and denote $\widetilde{d}=(\widetilde{d}_1,\ldots,\widetilde{d}_k)$. Notice that $\eta:\V\rightarrow V$ is surjective onto a neighborhood of $(0,-\widetilde{d}(Z_0))\in V$.
	Now, we present the main result of this section which is an extension of the Proposition \ref{ngeral_prop}.
	
	\begin{prop} \label{ngeral+_prop}
		Let $\Gamma_0$ be a $\s$-polycycle  of $Z_0=(X_0,Y_0)\in\Or$ having $k$ regular-tangential singularities $p_i\in\s$, of multiplicity $n_i$, $i=1,\ldots,k,$ satisfying property $(A)$. Then, the following statements hold.
		\begin{enumerate}[i)]
			\item There exists an annulus $\mathcal{A}_0$ at $\Gamma_0$ and a neighborhood $\V$ of $Z_0$ such that each $Z\in\V$ has at most one crossing limit cycle bifurcating from $\Gamma_0$ in $\mathcal{A}_0$, which is hyperbolic attracting.
			\item Let $Z_{\beta,\widetilde{\lambda}}$ be a continuous $2k$-parameter family in $\V$ such that $Z_{0,-\widetilde{d}(Z_0)}=Z_0$ and satisfying $Z_{\beta,\widetilde{\lambda}}\in\eta^{-1}(\beta,\widetilde{\lambda}),$ for every $(\beta,\widetilde{\lambda})\in V$. Then, for each $\widetilde{\lambda}$ near $-\widetilde{d}(Z_0)$ and for each connected component  $C_i$ of $\sigma_i(Z_{\beta,\widetilde{\lambda}})\cap \mathcal{A}_0$, there exist non-empty open intervals $I_{\widetilde{\lambda},C_i}$, satisfying $I_{\widetilde{\lambda},C_1}\times \cdots\times I_{\widetilde{\lambda},C_k}\times\{\widetilde{\lambda}\}\subset V$, such that $Z_{\bg,\lambda}$ has a hyperbolic attracting crossing limit cycle passing through $C_1 \times \cdots \times C_k$, for each $\bg\in I_{\widetilde{\lambda},C_1}\times \cdots\times I_{\widetilde{\lambda},C_k}$.
		\end{enumerate}	
	\end{prop}		
	\begin{proof}
		As seen before, there exists a neighborhood $\mathcal{V}$ of $Z_0$ in $\Or$ such that, for each $Z=(X,Y)\in\V$, we associate the displacement functions $\Delta_i(Z)$, $i=1,\ldots,k$, given in \eqref{dispsystem+}, which can be extended to $\widetilde{\Delta}_i(Z):(-\e,\e)\rightarrow \R$ (see Section \ref{disp_sec}).

		Define the $\Cr$ function $\mathcal{F}: \V\times V\times (-\e,\e)^k\rightarrow \R^k$ as
		$$\mathcal{F}(Z,\beta,\widetilde{\lambda},x)=\big( \mathcal{F}_1(Z,\beta,\widetilde{\lambda},x),\ldots,\mathcal{F}_k(Z,\beta,\widetilde{\lambda},x) \big),$$   
		where $V$ is an open neighborhood of $(0,-\widetilde{d}(Z_0))\in\R^{k}\times\R^{k}$ and, for $i=1,\ldots,k$, 
		$$\mathcal{F}_i(Z,\beta,\widetilde{\lambda},x)=\widetilde{\Delta}_i(Z)(h_Z^{i}(x_{i}),h_Z^{i+1}(x_{i+1}))-\beta_{i}(Z)-\widetilde{\lambda}_{i}(Z) x_{i+1}+\bg_i+\widetilde{\lambda}_i x_{i+1},$$ with $x_{k+1}=x_1$ and $h^{k+1}_Z=h^1_Z$.
		
		Notice that $\mathcal{F}(Z_0,0,-\widetilde{d}(Z_0), 0)=(0,\ldots,0) $ and
		$$D_x\mathcal{F}(Z_0,0,-\widetilde{d}(Z_0), 0)=\left(\begin{array}{ccccc}
		0 &-\widetilde{d}_2(Z_0) &  0 & \cdots  & 0 \\
		0 & 0 & -\widetilde{d}_3(Z_0) &  \cdots  & 0 \\
		\vdots & \vdots & \vdots & \ddots  &   \vdots\\
		0 & 0 & 0 & \cdots  & -\widetilde{d}_k(Z_0) \\
		-\widetilde{d}_1(Z_0) & 0  & 0& \cdots &  0
		\end{array}\right).$$
		
		From the Implicit Function Theorem for Banach Spaces and reducing $\V$ and $V$ if necessary, there exists a unique $\Cr$ function $\mathcal{X}:\mathcal{V}\times V\rightarrow (-\e,\e)^k$ such that $\mathcal{F}(Z,\beta,\widetilde{\lambda},x)=0$ if, and only if, $x=\mathcal{X}(Z,\beta,\widetilde{\lambda})$. Since 
		$$\mathcal{F}(Z,\bg,\widetilde{\lambda}, x)=\big(\bg_1+\widetilde{\lambda}_1 x_2,\ldots,\bg_{k-1}+\widetilde{\lambda}_{k-1} x_k,\bg_k+\widetilde{\lambda}_k x_{1} \big)+\er_2(x),$$
		it follows that $\mathcal{X}(Z,0,\widetilde{\lambda})=0$, for any $(Z,0,\widetilde{\lambda})\in\V\times V$. Consequently, 
		\begin{equation*}
			\mathcal{X}(Z,\bg,\widetilde{\lambda})=-\left(\dfrac{\bg_k}{\widetilde{\lambda}_k},\dfrac{\bg_1}{\widetilde{\lambda}_1},\ldots,\dfrac{\bg_{k-1}}{\widetilde{\lambda}_{k-1}}\right)+\er_2(\beta).
		\end{equation*}
		
		From the definition of the function $\mathcal{F}$, the unique zero of $\widetilde{\Delta}(Z)=\big(\widetilde{\Delta}_1(Z),\ldots,\widetilde{\Delta}_k(Z) \big)$ in $(-\e,\e)^k$ is given by
		\begin{equation}\label{zero}
			\mathcal{X}^*(Z)=\mathcal{X}(Z,\bg(Z),\widetilde{\lambda}(Z)).
		\end{equation}
		Hence, system \eqref{dispsystem+} has at most one zero in $\sigma_1(Z)\times\cdots\times\sigma_k(Z)$. Moreover, since 
		\[
		\dfrac{\partial \widetilde\Delta_i(Z_0)}{\partial x_{i+1}}(\mathcal{X}^*(Z_0))=-\widetilde d_i(Z_0)>0,
		\]
		it follows that
		\begin{equation}\label{derdelta}
		\dfrac{\partial \widetilde\Delta_i(Z)}{\partial x_{i+1}}(\mathcal{X}^*(Z))=\widetilde{\lambda}_i(Z)+\er_1(\mathcal{X}^*(Z))>0, \quad i=1,\ldots,k,
		\end{equation}
		for $Z$ sufficiently near $Z_0$.  
	
	Now, for $Z\in\V$ suppose that the solution $\mathcal{X}^*(Z)=(x_{1}^*(Z),\ldots,x_{k}^*(Z))\in\textrm{int}(\sigma_1(Z)\times\cdots\times\sigma_k(Z))$ of system \eqref{dispsystem+} is associated with a crossing limit cycle of $Z.$ From the Implicit Function Theorem,  for each $x_1$ sufficiently close to $x_{1}^*(Z)$ the orbit of $Z,$ starting at $(x_1,0)\in\sigma_1(Z)\times\{0\},$ intersects each $\textrm{int}(\sigma_i(Z))\times\{0\}$ at $(\xi_i(x_1),0)$ with $\xi_i(x_1)$ near $x_i^*(Z).$ Notice that $$\widetilde{\Delta}_i(Z)(h_Z^{i}(\xi_{i}(x_1)),h_Z^{i+1}(\xi_{i+1}(x_1)))=0,$$ for $i=1,\ldots,k-1.$ Consequently, 
\[
x_1\mapsto \widetilde{\Delta}_k(Z)(h_Z^{k}(\xi_{k}(x_1)),h_Z^{1}(x_1))
\]
is the displacement function associated with the crossing limit cycle defined in neighborhood of $x_1^*(Z)$ in $\textrm{int}(\sigma_1(Z))\times\{0\}$.  Clearly, the above displacement function vanishes at $x_1^*(Z)$. Moreover, from \eqref{derdelta}, the derivative of displacement function at $x_1^*$ is positive. Therefore, when the crossing limit cycle exists, it is hyperbolic and attracting.

		The proof of item $(i)$ follows by taking $\mathcal{A}_0=\{p\in M; d(p,\Gamma_0)<\e\}$, where $d$ denotes the Hausdorff distance.
		
		Now, consider the family $Z_{\bg,\widetilde{\lambda}}$ given in item $(ii)$. The unique zero of $\widetilde{\Delta}(Z_{\bg,\widetilde{\lambda}})$ is given by
		\begin{equation}\label{xstar+}
			\mathcal{X}^*(Z_{\beta,\widetilde{\lambda}})=-\left(\dfrac{\bg_k}{\widetilde{\lambda}_k},\dfrac{\bg_1}{\widetilde{\lambda}_1},\ldots,\dfrac{\bg_{k-1}}{\widetilde{\lambda}_{k-1}}\right)+\er_2(\bg).
		\end{equation}
		Recall that each isolated solution $x^*$ of system \eqref{dispsystem+} represents either a crossing limit cycle (if  $x^*\in\textrm{int}(\sigma_1(Z_{\bg,\widetilde{\lambda}})\times\cdots\times\sigma_k(Z_{\bg,\widetilde{\lambda}}))$) or a $\s$-polycycle (if $x^*\in\partial(\sigma_1(Z_{\bg,\widetilde{\lambda}})\times\cdots\times\sigma_k(Z_{\bg,\widetilde{\lambda}}))$). 
		So, for $i=1,\ldots,k$, let $C_i$ be a connected component of $\sigma_i(Z_{\bg,\widetilde{\lambda}})\subset(-\e,\e)$ for some fixed parameter $\widetilde{\lambda}\in\pi_{2}(V).$ Hence, from \eqref{xstar+}, there exists a non-empty open interval $I_{\widetilde{\lambda},C_i}$  such that $I_{\widetilde{\lambda},C_1}\times\ldots\times I_{\widetilde{\lambda},C_i}\times\{\widetilde{\lambda}\}\subset V$ and $\mathcal{X}^*(Z_{\beta,\widetilde{\lambda}})\in \textrm{int}(C_1\times\ldots\times C_k)$ whenever $\bg_i\in I_{\widetilde{\lambda},C_i}$.	
	\end{proof}
	
	\begin{remark}\label{remori}
		Regarding Propositions \ref{stab+_prop} and \ref{ngeral+_prop}, if we change the orientation in property $(A)$ in order to reverse the orientation of the $\s$-polycycle, all the results remain the same reversing the stability the   $\s$-polycycle and the crossing limit cycle.
	\end{remark}
	
	These results are illustrated in the next section for the case where the $\s$-polycycle has two fold-regular singularities.

\subsection{$\s$-Polycycles having two regular-fold singularities}\label{2fold-fold}

\

Firstly, without loss of generality, we assume some conditions in order to characterize the nonsmooth vector fields $Z_0=(X_0,Y_0)\in\Or$ which admit a $\s$-polycycle $\Gamma_0$ satisfying (A) and having only two regular-fold singularities (see Figure \ref{2FRcycle}). So, consider a coordinate system $(x,y)$ such that $x(p_1)=a_1,$ $y(p_1)=0$, $x(p_2)=a_2>0$, $y(p_2)=0$ and $h(x,y)=y$ in neighborhoods of $p_1$ and $p_2$. Consider the following sets of hypotheses:

\begin{itemize}
	\item [(DRF-A):]\begin{itemize}
		\item[$-$] $p_1$ is a visible regular-fold singularity of $X_0$ and $\pi_{1}\circ X_0(p_1)>0$;
		\item[$-$] $p_2$ is a visible regular-fold singularity  of $Y_0$ and $\pi_{1}\circ Y_0(p_2)<0$;			
		\item [$-$] $W^{u}_{+}(p_1)$ reaches $\s$ transversally at $p_2$;
		\item[$-$] $W^{u}_{-}(p_2)$ reaches $\s$ transversally at $p_1$
		\item[ ] 	
	\end{itemize}
	\item [(DRF-B):]\begin{itemize}
		\item[$-$] $p_1$ is a visible regular-fold singularity of $X_0$ and $\pi_{1}\circ X_0(p_1)<0$;
		\item[$-$] $p_2$ is a visible regular-fold singularity of $Y_0$ and $\pi_{1}\circ Y_0(p_2)>0$;			
		\item [$-$] $W^{u}_{+}(p_1)$ reaches $\s$ transversally at $p_2$;
		\item[$-$] $W^{u}_{-}(p_2)$ reaches $\s$ transversally at $p_1$
	\end{itemize}			
\end{itemize}

\begin{figure}[h]
	\centering
	\bigskip
	\begin{overpic}[width=10cm]{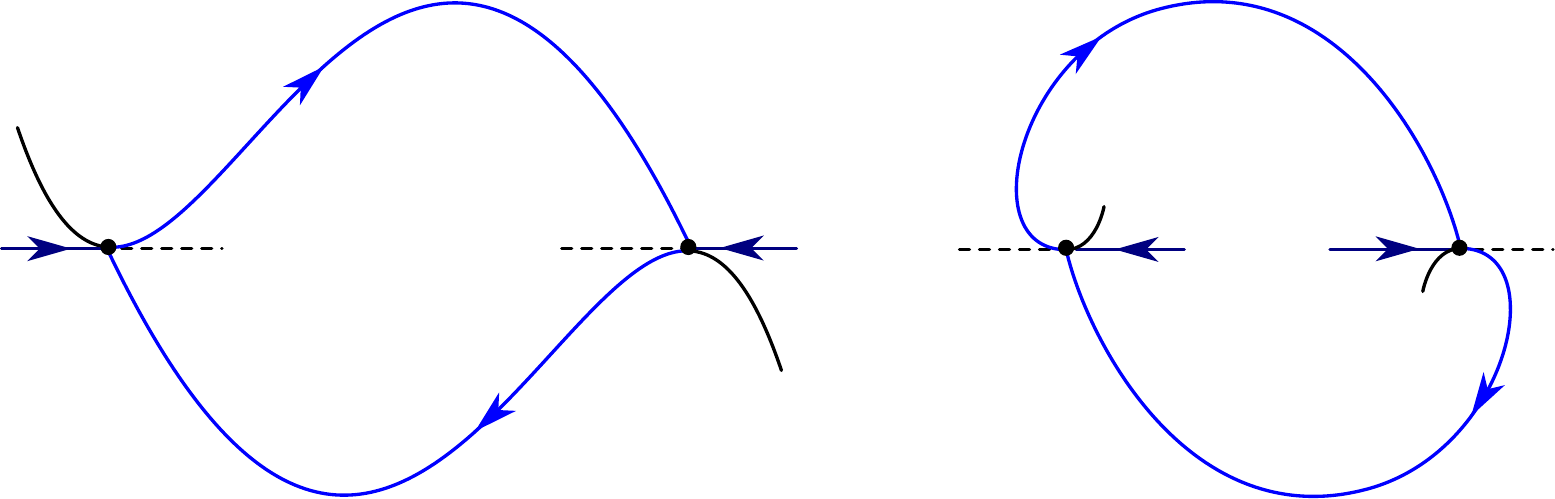}
		\put(23,-5){{\footnotesize $(a)$}}		
		\put(6,18){{\footnotesize $p_1$}}
		\put(67,18){{\footnotesize $p_1$}}
		\put(45,18){{\footnotesize $p_2$}}
		\put(94,18){{\footnotesize $p_2$}}		
		\put(83,-5){{\footnotesize $(b)$}}	
	\end{overpic}
	\bigskip
	\caption{ $\s$-polycycle $\Gamma_0$ of $Z_0$ under the set of hypotheses $(a)$ (DRF-A) and $(b)$ (DRF-B), respectively. }	\label{2FRcycle}
\end{figure}

Hypotheses (DRF-A) and (DRF-B) fix the orientation and the stability of the $\s$-polycycle $\Gamma_0$. Indeed, in this case $\Gamma_0$ is C-attracting. According to Remark \ref{remori}, the stability of $\Gamma_0$ is reversed if we change the orientation.

Here we shall assume that  $Z_0$ satisfies (DRF-A), the case (DRF-B) will follow analogously. In this case, $Z_0$ admits a $\s$-polycycle $\Gamma_0$ given by the union  $W^{u}_{+}(a_1,0)\cup W^{u}_{-}(a_2,0)\cup \{(a_1,0), (a_2,0) \}$. We shall see that  $\mathcal{S}(\Gamma_0)=2$.

Since regular-fold singularities are locally structurally stable, they persist under small perturbations. Consequently, without loss of generality, we may assume that the diffeomorphisms $h_Z^i,$ $i=1,2,$ provenient from Theorem \ref{charac2} may be taken as the identity. Accordingly, the displacement functions write
\[
\begin{array}{rl}
\Delta_1(x_1,x_2)=& T_1^u(x_1)-[D_1]^{-1}\circ T_2^s(x_2)\vspace{0.2cm}\\
=&\la_0^1+\kappa_1(x_1-a_1)^2+\er_3(x_1-a_1)\vspace{0.1cm}\\
&-\widetilde c_1-\widetilde d_1(x_2-a_2)+\er_2(x_2-a_2),\vspace{0.3cm}\\

\Delta_2(x_2,x_1)=& T_2^u(x_2)-[D_2]^{-1}\circ T_1^s(x_1)\vspace{0.2cm}\\
=&\la_0^2+\kappa_2(x_2-a_2)^2+\er_3(x_2-a_2)\vspace{0.1cm}\\
&-\widetilde c_2-\widetilde d_2(x_1-a_1)+\er_2(x_1-a_1),\\
\end{array}
\]
where $\kappa_1<0,$ $\kappa_2>0,$ and $\widetilde d_i>0,$ for $i=1,2.$ Therefore, denoting $\beta_i=\la_0^i-\widetilde c_i,$ $i=1,2,$ (see Figure \ref{splitting}) the crossing system \eqref{algsystem} becomes
\begin{equation}\label{system1}
\left\{\begin{array}{l}
\beta_1-\widetilde d_1\xi_2+\kappa_1\xi_1^2+\er_2(\xi_2)+\er_3(\xi_1)=0,\vspace{0.1cm}\\
\beta_2-\widetilde d_2\xi_1+\kappa_2\xi_2^2+\er_2(\xi_2)+\er_3(\xi_1)=0,\vspace{0.1cm}\\
\xi_i=x_i-a_i,\,i=1,2,\vspace{0.1cm}\\
(x_1,x_2)\in\sigma_1(Z)\times \sigma_2(Z)=[a_1,a_1+\e)\times(a_2-\e,a_2].
\end{array}\right.
\end{equation}

\begin{figure}[h!]
	\centering
	\bigskip
	\begin{overpic}[width=8cm]{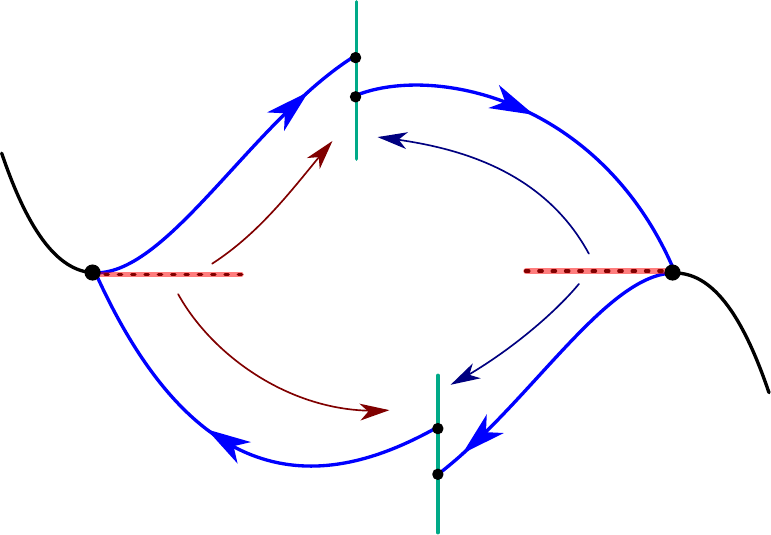}
		\put(11,36){{\footnotesize $p_1$}}
		\put(87.5,36){{\footnotesize $p_2$}}
		\put(42,66){{\footnotesize $\tau_1^u$}}
			\put(57.5,1){{\footnotesize $\tau_2^u$}}
		\put(42,55){{\scriptsize $\widetilde{c}_1$}}
		\put(47,60.5){{\scriptsize $\lambda_0^1$}}
				\put(58,14){{\scriptsize $\widetilde{c}_2$}}
		\put(51,6.5){{\scriptsize $\lambda_0^2$}}	
	\end{overpic}
	\bigskip
	\caption{ Splitting of the separatrices for a perturbed system $Z\in\V$. }	\label{splitting}
\end{figure}

In what follows we use the  crossing system \eqref{system1} to describe the bifurcation diagram of $Z_0$ at $\Gamma_0$ assuming the set of hypotheses (DRF-A) (see Figure \ref{bif2fr1}).

\begin{theorem}\label{bifdig_foldregular}
	Let $Z_{0}$ be a nonsmooth vector field having a $\s$-polycycle satisfying the set of hypotheses (DRF-A). Therefore, there exists an annulus $\mathcal{A}_0$ around $\Gamma_0$ such that for each annulus $\mathcal{A}$, with $\Gamma_0\subset\mathcal{A}\subset \mathcal{A}_0,$ there exist neighborhoods $\V\subset\Or$ of  $Z_0$ and $V\subset\R^2$ of $(0,0)$, a surjective function $(\bg_1,\bg_2): \V \rightarrow V$ with $(\bg_1,\bg_2)(Z_0)=(0,0)$, and two smooth functions $\gamma_{1},\gamma_{2}:\V\rightarrow \R$ with $\gamma_1(Z_0)=\gamma_2(Z_0)=0,$ for which the following statements hold inside $\mathcal{A}$.
	\begin{enumerate}
		\item If $\beta_2>0$ and $\beta_1>\gamma_{1}$, then $Z$  has a sliding cycle containing the regular-fold singularity $p_2$ and a unique sliding segment.
		\item If $\beta_2>0$ and $\beta_1=\gamma_{1}$, then $Z$ has a C-attracting $\s$-polycycle containing the regular-fold singularity $p_2$.
		\item If $\beta_2>0$ and $0<\beta_1<\gamma_{1}$, then $Z$ has a hyperbolic attracting crossing limit cycle.
		\item If $\beta_2>0$ and $\beta_1=0$, then $Z$ has a hyperbolic attracting crossing limit cycle and a heteroclinic connection between $p_1$ and $p_2$.
		\item If $\beta_2>0$ and $\beta_1<0$, then $Z$ has a hyperbolic attracting crossing limit cycle.
		\item If $\beta_2=0$ and $\beta_1<0$, then $Z$ has a hyperbolic attracting crossing limit cycle and a heteroclinic connection between $p_1$ and $p_2$.
		\item If $\beta_2=\beta_1=0$, then $Z$ has a C-attracting $\s$-polycycle containing two  regular-fold singularities.
		\item If $\beta_1<0$ and $\gamma_2<\beta_2<0$, then $Z$ has a hyperbolic attracting crossing limit cycle.
\item If $\beta_1<0$ and $\beta_2=\gamma_2$, then $Z$ has a C-attracting $\s$-polycycle containing the  regular-fold singularity $p_1$.
\item If $\beta_1<0$ and $\beta_2<\gamma_2$, then $Z$ has a sliding cycle containing the  regular-fold singularity $p_1$ and a unique sliding segment.
		\item If $\beta_2<0$ and $\beta_1=0$, then $Z$ has a sliding cycle containing two regular-fold singularities and one sliding segment.

		\item If $\beta_2<0$ and $\beta_1>0$, then $Z$ has a sliding cycle containing two regular-fold singularities and two sliding segments.

		\item If $\beta_2=0$ and $\beta_1>0$, then $Z$ has a sliding cycle containing two regular-fold singularities and one sliding segment.
	\end{enumerate}	
	Here,
	$$\gamma_{1}=-\dfrac{\kappa_1}{\widetilde d_2^2}\bg_2^2+\er_3(\bg_2) \textrm{ and } \gamma_{2}=-\dfrac{\kappa_2}{\widetilde d_1^2}\bg_1^2+\er_3(\bg_1).$$  
	In addition, in the cases $(1),$ and $(10)-(13),$ $Z$ does not admit limit cycles.
\end{theorem}

\begin{proof}
From the construction of the crossing system \eqref{algsystem}, performed in Section \ref{descripregtan}, we get the existence of an annulus $\mathcal{A}_0$ around $\Gamma_0$ and neighborhoods $\V_0\subset\Or$ of  $Z_0$ and $V_0\subset\R^2$ of $(0,0)$, for which the crossing system \eqref{system1} is well defined.

Now, given an annulus $\mathcal{A}$, with $\Gamma_0\subset\mathcal{A}\subset \mathcal{A}_0,$ let $\e>0$ satisfy $[a_1,a_1+\e)\times \{0\}\subset \mathcal{A}$ and $(a_2-\e,a_2]\times \{0\}\subset \mathcal{A}$. Consider the function $F:\V_0\times(-\e,\e)^2\times V_0 \rightarrow \rn{2}$ given by 
	\[F(Z,\xi_1,\xi_2,\bg_1,\bg_2)=(F_{1}(Z,\xi_1,\xi_2,\bg_1),F_2(Z,\xi_1,\xi_2,\bg_2)),\]
	where 
	$$\begin{array}{l}
	F_{1}(Z,\xi_1,\xi_2,\bg_1)= \widetilde{F_1}(Z,\xi_1,\xi_2)-\beta_1(Z) +\bg_1,\\
	F_{2}(Z,\xi_1,\xi_2,\bg_2)= \widetilde{F_2}(Z,\xi_1,\xi_2)-\beta_2(Z) +\bg_2,\\		
	\end{array}$$
	and $\widetilde{F_1}$ and $\widetilde{F_2}$ are given by the left-hand side of the first two equations of \eqref{system1}.
	
	Notice that $F(Z_0,0,0,0,0)=(0,0)$ and $$\det[ D_{(\xi_1,\xi_2)}F(Z_0,0,0,0,0)]=-\widetilde d_1(Z_0)\widetilde d_2(Z_0)\neq 0.$$
	From the Implicit Function Theorem for Banach Spaces, there exist neighborhoods $\V\subset\V$ and $V\subset V_0$  and unique $\Cr$ functions $\Xi_1,\Xi_2:\V\times V\rightarrow (-\e,\e)$ such that 
	\[
	F(Z,\Xi_1(Z,\bg_1,\bg_2),\Xi_2(Z,\bg_1,\bg_2),\bg_1,\bg_2)=(0,0).
	\] Consequently, for each $Z\in\V$, the crossing system \eqref{system1} has at most one solution. In fact, \eqref{system1} is satisfied if, and only if, 
	\begin{equation}\label{iff} 
	(\xi_1,\xi_2)=(\Xi_1(Z,\bg_1(Z),\bg_2(Z)),\Xi_2(Z,\bg_1(Z),\bg_2(Z)))\in[0,\e)\times(-\e,0].
	\end{equation}
	Therefore, each $Z\in\V$ has either a $\s$-polycycle having a unique regular-fold singularity (which occurs when $\xi_1=0$ or $\xi_2=0$) or at most one crossing limit cycle. 
	
	In what follows, we find parameters $(\bg_1(Z),\bg_2(Z))$ satisfying \eqref{iff}.
	
	First, $\Xi_2(Z,\bg_1(Z),\bg_2(Z))=0$ implies the existence of a $\s$-polycycle of $Z$ passing through the regular-fold singularity $p_2$. Applying the Implicit Function Theorem to $g(Z,\xi_1,\bg_2)=F_2(Z,\xi_1,0,\bg_2)$ at $(Z_0,0,0)$, we obtain the existence of a unique $\Cr$ function $\widetilde{\Xi}_1(Z,\bg_2)$ such that $g(Z,\widetilde{\Xi}_1(Z,\bg_2),\bg_2)=0$. In addition,
	$$\widetilde{\Xi}_1(Z,\bg_2)=\dfrac{\beta_2}{\widetilde{d}_2(Z)}+\er_2(\bg_2)=\er_1(\bg_2).$$
	Now, applying the Implicit Function Theorem to $h(Z,\bg_1,\bg_2)=F_1(Z,\widetilde{\Xi}_1(Z,\bg_2),0,\bg_1)$ at the point $(Z_0,0,0)$
	we obtain a function $\overline{\bg_1}(Z,\bg_2)$ such that $h(Z,\overline{\bg_1}(Z,\bg_2),\bg_2)=0$. It follows directly from the expression of $h$ that
	$$\overline{\bg_1}(Z,\bg_2)=-\dfrac{\kappa_1(Z)}{\widetilde{d}_2(Z)^2}\bg_2^2+\er_3(\bg_2).$$
	Hence, it shows that $F(Z,\widetilde{\Xi}_1(Z,\bg_2(Z)),0,\overline{\bg_1}(Z,\bg_2(Z)),\bg_2(Z))=(0,0)$. From uniqueness of the solution,
	\[\Xi_1(Z,\overline
{\bg_1}(Z,\bg_2(Z)),\bg_2(Z))=\widetilde{\Xi}_1(Z,\bg_2(Z))\quad \text{and}\quad \Xi_2(Z,\overline
{\bg_1}(Z,\bg_2(Z)),\bg_2(Z))=0.
\] Thus, $\Xi_2(Z,\bg_1(Z),\bg_2(Z))=0$ if, and only if, $\bg_1(Z)=\overline{\bg_1}(Z,\bg_2(Z))$. Moreover, since $\widetilde{d}_2(Z)>0$, it follows that $\widetilde{\Xi}_1(Z,\bg_2(Z))\in[0,\e)$ if, and only if, $\bg_2(Z)\geq 0$. Finally, defining $\gamma_{1}(Z)=\overline{\bg_1}(Z,\bg_2(Z))$, we have that each $Z\in\V,$ satisfying $\bg_1(Z)=\cg_1(Z)$ and $\bg_2(Z)\geq 0,$ has a $\s$-polycycle containing a unique regular-fold singularity, namely  $p_2=(a_2,0)$.
	
	Analogously, $\Xi_1(Z,\bg_1(Z),\bg_2(Z))=0$ implies the existence of a $\s$-polycycle of $Z$ passing through the regular-fold singularity $p_1$. Following the same ideas above, we obtain a unique $\Cr$ function $\widetilde{\Xi}_2(Z,\bg_1)$ such that $F_{1}(Z,0,\widetilde{\Xi}_2(Z,\bg_1),\bg_1)=0$. Furthermore
	$$\widetilde{\Xi}_2(Z,\bg_1)=\dfrac{\beta_1}{\widetilde{d}_1(Z)}+\er_2(\bg_1).$$
	Also, we obtain a unique $\Cr$ function $\overline{\bg_2}(Z,\bg_1)$ such that
	$$F_2(Z,0,\widetilde{\Xi}_2(Z,\bg_1),\overline{\bg_2}(Z,\bg_1))=0\quad\text{and}\quad\overline{\bg_2}(Z,\bg_1)=-\dfrac{\kappa_2(Z)}{\widetilde{d}_1(Z)^2}\bg_1^2+\er_3(\bg_1).$$
	Therefore, $F(Z,0,\widetilde{\Xi}_2(Z,\bg_1(Z)),\bg_1(Z),\overline{\bg_2}(Z,\bg_1(Z)))=(0,0)$. Again, from uniqueness of the solution, it follows that 
	\[
	\begin{array}{l}
	\Xi_1(Z,\bg_1(Z),\overline
	{\bg_2}(Z,\bg_1(Z)))=0\quad \text{and}\\
	 \Xi_2(Z,\bg_1(Z),\overline
	{\bg_2}(Z,\bg_1(Z)))=\widetilde{\Xi}_2(Z,\bg_1(Z)).
	\end{array}
	\] Hence, $\Xi_1(Z,\bg_1(Z),\bg_2(Z))=0$ if, and only if, $\bg_2(Z)=\overline{\bg_2}(Z,\bg_1(Z))$. Also, since $\widetilde{d}_1(Z)>0$, it follows that $\widetilde{\Xi}_2(Z,\bg_1(Z))\in(-\e,0]$ if, and only if, $\bg_1(Z)\leq 0$. Defining $\gamma_{2}(Z)=\overline{\bg_2}(Z,\bg_1(Z))$, we have that each $Z\in\V$ satisfying $\bg_2(Z)=\cg_2(Z)$ and $\bg_1(Z)\leq 0$ has a $\s$-polycycle containing a unique regular-fold singularity given by $p_1=(a_1,0)$.
	
	The C-attractiveness  of the  $\s$-polycycle detected above is given by Proposition \ref{stab+_prop}. Hence, items $(2),(7)$ and $(9)$ are proved. 
	
	In what follows we shall identify when the solution $\big(\Xi_1(Z,\bg_1(Z),\bg_2(Z)),$ $\Xi_2(Z,\bg_1(Z),\bg_2(Z))\big)$ of the crossing system \eqref{system1} corresponds to a crossing limit cycle.

	Note that
	\begin{equation}
	\label{expxi}
		\begin{array}{l}
	\Xi_1(Z,\bg_1(Z),\bg_2(Z))=\dfrac{1}{\widetilde{d}_2(Z)}\bg_2(Z)+\er_2(\bg_1(Z),\bg_2(Z)),\vspace{0.2cm}\\
	\Xi_2(Z,\bg_1(Z),\bg_2(Z))=\dfrac{1}{\widetilde{d}_1(Z)}\bg_1(Z)+\er_2(\bg_1(Z),\bg_2(Z)).
	\end{array}
	\end{equation}
	Recall that $\Xi_2(Z,\cg_1(Z),\bg_2(Z))=0$. Using \eqref{expxi}, we expand $\Xi_2(Z,\bg_1(Z),\bg_2(Z))$ around $\bg_1(Z)=\cg_1(Z)$ as
	$$
	\begin{array}{rl}
	\Xi_2(Z,\bg_1(Z),\bg_2(Z))=&\left(\dfrac{1}{\widetilde{d}_1(Z)}+\er_1(\bg_2(Z))\right)(\bg_1(Z)-\gamma_1(Z))\vspace{0.2cm}\\
	&+ \er_2(\bg_1(Z)-\gamma_1(Z)).
	\end{array}
	$$
	Since $\widetilde{d}_1(Z)>0$, it follows that $\Xi_2(Z,\bg_1(Z),\bg_2(Z))\in (-\e,0)$ if, and only if, $\bg_1(Z)<\gamma_1(Z)$. Also, $\Xi_1(Z,\gamma_{1}(Z),\bg_2(Z))\in(0,\e)$ for $\bg_2(Z)>0$ and, thus, $\Xi_1(Z,\beta_1(Z),\bg_2(Z))\in(0,\e)$ for $\bg_2(Z)>0$ and $\beta_1(Z)$ sufficiently close to $\gamma_1(Z)$. Finally, we conclude that $(\Xi_1(Z,\bg_1(Z),\bg_2(Z)),$ $\Xi_2(Z,\bg_1(Z),\bg_2(Z)))\in(0,\e)\times(-\e,0)$ with $\bg_2(Z)>0$ if, and only if, $\bg_1(Z)<\cg_1(Z)$. Hence, we get the existence or not of crossing limit cycles in items $(1),(3),(4),$ and $(5)$.

	Analogously, since $\Xi_1(Z,\bg_1(Z),\cg_2(Z))=0$,  the expansion of $\Xi_1(Z,\bg_1(Z),\bg_2(Z))$ around $\bg_2(Z)=\cg_2(Z)$ writes
	$$
	\begin{array}{rl}
	\Xi_1(Z,\bg_1(Z),\bg_2(Z))=&\left(\dfrac{1}{\widetilde{d}_2(Z)}+\er_1(\bg_1(Z))\right)(\bg_2(Z)-\cg_2(Z))\vspace{0.2cm}\\
	&+ \er_2(\bg_2(Z)-\cg_2(Z)).
	\end{array}
	$$
	Recalling that $\widetilde{d}_2(Z)>0$, we obtain $\Xi_1(Z,\bg_1(Z),\bg_2(Z))\in (0,\e)$ if, and only if, $\bg_2(Z)>\cg_2(Z)$. Also, $\Xi_2(Z,\bg_{1}(Z),\cg_2(Z))\in(-\e,0)$ for $\bg_1(Z)<0$. Therefore, $\Xi_2(Z,\bg_1(Z),\bg_2(Z))\in(-\e,0)$, for $\bg_1(Z)<0$ and $\bg_2(Z)$ sufficiently close to $\cg_2(Z)$. Finally, we conclude that $(\Xi_1(Z,\bg_1(Z),\bg_2(Z)),$ $\Xi_2(Z,\bg_1(Z),\bg_2(Z)))\in(0,\e)\times(-\e,0)$ with $\bg_1(Z)<0$ if, and only if, $\bg_2(Z)>\cg_2(Z)$. Hence, we get the existence or not of crossing limit cycles in items $(6)$, $(8)$ and $(10).$

	Now, notice that 
\[\begin{array}{l}
\Xi_1(Z,0,\bg_2(Z))=\dfrac{1}{\widetilde{d}_2(Z)}\bg_2(Z)+\er_2(\bg_2(Z)),\vspace{0.2cm}\\
\Xi_2(Z,\bg_1(Z),0)=\dfrac{1}{\widetilde{d}_1(Z)}\bg_1(Z)+\er_2(\bg_1(Z)).
\end{array}
\]
Therefore, $\Xi_1(Z,0,\bg_2(Z))<0$ and $\Xi_2(Z,0,\bg_2(Z))>0$, provided that $\bg_2(Z)<0$ and $\bg_1(Z)>0$. This  means that \eqref{system1} has no solutions when $\bg_1(Z)=0$ and $\bg_2(Z)<0$ or $\bg_2(Z)=0$ and $\bg_1(Z)>0$. From continuity, if follows that $\Xi_1(Z,\bg_1(Z),\bg_2(Z))\in(-\e,0)\times(0,\e)$ for $\bg_1(Z)>0$ and $\bg_2(Z)<0$. Hence, we conclude the non-existence of crossing limit cycles in items $(11)$, $(12)$ and $(13)$.

Notice that $\bg_1(Z)= T_1^u(Z)(a_1)-[D_1(Z)]^{-1}\circ T_2^s(Z)(a_2)$ and $\bg_2(Z)=T_2^u(Z)(a_2)-[D_2(Z)]^{-1}\circ T_1^s(Z)(a_1)$. Heteroclinic connections exist when $\bg_1(Z)=0$ or $\bg_1(Z)=0.$ If either $\bg_1(Z)=0$ and $\bg_2(Z)>0$ or $\bg_1(Z)<0$ and $\bg_2(Z)=0,$ the heteroclinic connection is not contained in a sliding cycle. This correspond to items $(4)$ and $(6).$
	
Finally, the sliding region corresponding to $Z$ is given by $\s^s=(a_1-\e,a_1)\times\{0\}\cup (a_2,a_2-\e)\times\{0\}$, for every $Z\in\V$, the sliding vector field $F_{Z}$ is regular in $\s^s$,  $\pi_1\circ F_Z(a_1,0)>0,$ and $\pi_1\circ F_Z(a_2,0)<0$. Therefore, the sliding phenomena detected in items $(1)$ and $(10)-(13)$ follows straightforwardly. Hence, the proof is concluded.			
\end{proof}

\begin{remark}
	We notice that the set of displacement functions associated with a nonsmooth vector field $Z_0$ at a $\s$-polycycle satisfying the hypotheses (DRF-B) generates the same system of equations \eqref{system1} obtained for the case (DRF-A). Nevertheless, the domain $\sigma_1\times\sigma_2$ will be given by $\sigma_1\times\sigma_2=(a_1-\e,a_1]\times[a_2,a_2+\e)$. The bifurcation diagram of $Z_0$ can be obtained analogously and has the same structure and objects of the case (DRF-A). Therefore, we shall omit it here. 
	\end{remark}

\section*{Acknowledgements}

The authors are very grateful to Marco A. Teixeira, who suggested the problem, for meaningful discussions and constructive criticism on the manuscript.

KSA is partially supported by Coordena\c{c}\~{a}o de Aperfei\c{c}oamento de Pessoal de N\'{i}vel Superior (CAPES) grant 88887.308850/2018-00.
OMLG is partially supported by S\~{a}o Paulo Research Foundation (FAPESP) grant 2015/22762-5. DDN is partially supported by S\~{a}o Paulo Research Foundation (FAPESP) grants 2022/09633-5, 2019/10269-3, and 2018/13481-0, and by Conselho Nacional de Desenvolvimento Cient\'{i}fico e Tecnol\'{o}gico (CNPq) grants 309110/2021-1.
All the authors are partially supported by Conselho Nacional de Desenvolvimento Cient\'{i}fico e Tecnol\'{o}gico (CNPq) grant 438975/2018-9.
	  
	  \bibliographystyle{abbrv}
\bibliography{references.bib}

\begin{thebibliography}{10}

\bibitem{A16}
K.~S. Andrade.
\newblock {\em On degenerate cycles in discontinuous vector fields and the
  Dulac's problem}.
\newblock Thesis (Unicamp), 2016.

\bibitem{AJMT17}
K.~S. Andrade, M.~R. Jeffrey, R.~M. Martins, and M.~A. Teixeira.
\newblock Homoclinic boundary-saddle bifurcations in planar nonsmooth vector
  fields.
\newblock {\em International Journal of Bifurcation and Chaos}, 32(04), Mar.
  2022.

\bibitem{BPET18}
L.~Benadero, E.~Ponce, A.~El~Aroudi, and F.~Torres.
\newblock Limit cycle bifurcations in resonant {LC} power inverters under zero
  current switching strategy.
\newblock {\em Nonlinear Dynamics}, 91(2):1145--1161, 2018.

\bibitem{BS16}
C.~Bonet-Rev\'{e}s and T.~M-Seara.
\newblock Regularization of sliding global bifurcations derived from the local
  fold singularity of {F}ilippov systems.
\newblock {\em Discrete Contin. Dyn. Syst.}, 36(7):3545--3601, 2016.

\bibitem{BCT12}
C.~A. Buzzi, T.~Carvalho, and M.~A. Teixeira.
\newblock On three-parameter families of {F}ilippov systems- the fold-saddle
  singularity.
\newblock {\em International Journal of Bifurcation and Chaos}, 22(12):1250291,
  2012.

\bibitem{Fei1}
M.~Di~Bernardo, M.~I. Feigin, S.~J. Hogan, and M.~E. Homer.
\newblock Local analysis of {C}-bifurcations in n-dimensional piecewise-smooth
  dynamical systems.
\newblock {\em Chaos, Solitons and Fractals: the interdisciplinary journal of
  Nonlinear Science, and Nonequilibrium and Complex Phenomena},
  11(10):1881--1908, 1999.

\bibitem{Fei2}
M.~I. Feigin.
\newblock On the structure of {C}-bifurcation boundaries of
  piecewise-continuous systems.
\newblock {\em Journal of Applied mathematics and Mechanics}, 42(5):885--895,
  1978.

\bibitem{F}
A.~F. Filippov.
\newblock {\em Differential equations with discontinuous righthand sides:
  control systems}, volume~18.
\newblock Springer Science \& Business Media, 1988.

\bibitem{FPT15}
E.~Freire, E.~Ponce, and F.~Torres.
\newblock On the critical crossing cycle bifurcation in planar {F}ilippov
  systems.
\newblock {\em Journal of Differential Equations}, 259(12):7086--7107, 2015.

\bibitem{GTS}
M.~Guardia, T.~M. Seara, and M.~A. Teixeira.
\newblock Generic bifurcations of low codimension of planar filippov systems.
\newblock {\em Journal of Differential Equations}, 250(4):1967--2023, 2011.

\bibitem{H}
P.~Hartman.
\newblock On local homeomorphisms of euclidean spaces.
\newblock {\em Bol. Soc. Mat. Mexicana}, 5(2):220--241, 1960.

\bibitem{KRG}
Y.~A. Kuznetsov, S.~Rinaldi, and A.~Gragnani.
\newblock One-parameter bifurcations in planar {F}ilippov systems.
\newblock {\em International Journal of Bifurcation and chaos},
  13(08):2157--2188, 2003.

\bibitem{LH13}
F.~Liang and M.~Han.
\newblock The stability of some kinds of generalized homoclinic loops in planar
  piecewise smooth systems.
\newblock {\em International Journal of Bifurcation and Chaos}, 23(02):1350027,
  2013.

\bibitem{LW17}
F.~Liang and D.~Wang.
\newblock Limit cycle bifurcations near a piecewise smooth generalized
  homoclinic loop with a saddle-fold point.
\newblock {\em International Journal of Bifurcation and Chaos}, 27(05):1750071,
  2017.

\bibitem{LR14}
Y.~Liu and V.~G. Romanovski.
\newblock Limit cycle bifurcations in a class of piecewise smooth systems with
  a double homoclinic loop.
\newblock {\em Applied Mathematics and Computation}, 248:235--245, 2014.

\bibitem{NR21}
D.~D. Novaes and G.~Rond\'{o}n.
\newblock Smoothing of nonsmooth differential systems near regular-tangential
  singularities and boundary limit cycles.
\newblock {\em Nonlinearity}, 34(6):4202--4263, 2021.

\bibitem{NR22}
D.~D. Novaes and G.~Rond\'{o}n.
\newblock On limit cycles in regularized {F}ilippov systems bifurcating from
  homoclinic-like connections to regular-tangential singularities.
\newblock {\em Phys. D}, 442:Paper No. 133526, 15, 2022.

\bibitem{NTZ18}
D.~D. Novaes, M.~A. Teixeira, and I.~O. Zeli.
\newblock The generic unfolding of a codimension-two connection to a two-fold
  singularity of planar {F}ilippov systems.
\newblock {\em Nonlinearity}, 31(5):2083, 2018.

\bibitem{Wu22}
F.~Wu, L.~Huang, and J.~Wang.
\newblock Bifurcation of the critical crossing cycle in a planar piecewise
  smooth system with two zones.
\newblock {\em Discrete and Continuous Dynamical Systems - Series B},
  27(9):5047--5083, 2022.

\end{thebibliography}

\end{document}